\documentclass[12pt,epsfig,amsfonts]{amsart} 
\usepackage{amsmath,amsthm,amssymb,amscd,epsfig}
\usepackage{graphicx}

\newtheorem*{theorema}{Theorem A}
\newtheorem*{theoremb}{Theorem B}

\newtheorem{cor}{Corollary}
\newtheorem{lemma}{Lemma}[section]
\newtheorem{sublemma}[lemma]{Sublemma}

\newtheorem{prop}[lemma]{Proposition}

\begin{document}
\author{Yong Moo Chung* and Hiroki Takahasi**}

%\address{Department of Applied Mathematics, Hiroshima University,
%Higashi-Hiroshima 739-8527, JAPAN} \email{chung@amath.hiroshima-u.ac.jp}
%\address{FIRST, Aihara Innovative Mathematical Modelling Project,
%Japan Science and Technology Agency; Institute of Industrial Science, University of Tokyo, Tokyo
%153-8505, JAPAN} \email{h{\_}takahasi@sat.t.u-tokyo.ac.jp}
\title[Multifractal formalism
for Benedicks-Carleson  quadratic maps] {Multifractal formalism for Benedicks-Carleson quadratic maps}
%\date{\today}
\thanks{*Department of Applied Mathematics, Hiroshima University,
Higashi-Hiroshima 739-8527, JAPAN chung@amath.hiroshima-u.ac.jp
**Department of Electronic Science and Engineering,
Kyoto University, Kyoto
606-8501, JAPAN takahasi.hiroki.7r@kyoto-u.ac.jp}
\begin{abstract}
For a positive measure set of nonuniformly expanding
quadratic maps on the interval we effect a multifractal formalism, i.e.,
  decompose the phase space into level sets
of time averages of a given continuous function and consider the
associated {\it Birkhoff spectrum} which encodes this 
decomposition. We derive a formula which relates the Hausdorff dimension of level sets to
entropies and Lyapunov exponents of invariant probability measures, and then use this formula to
show that the spectrum is continuous. In order to estimate the Hausdorff dimension from above,
one has to ``see" sufficiently many points. To this end, 
%we use the sub-exponential slow recurrence condition of Benedicks
%$\&$ Carleson and 
we construct a family of towers. 
Using these towers we 
establish a large
deviation principle of empirical distributions, with Lebesgue as
a reference measure.
\end{abstract}

\maketitle

%\noindent{\small {\bf R\'esume.} Pour des applications stochastiques quadratiques de l'intervalle correspondent
%\'a l'ensemble des param\'etres de mesure positif, nous demontrons le formalisme multifractal, c'est-\`a-dir\'e,
%divisons l'espace de phase dans des niveau de moyenne de temps d'une fonction 
%observable don\'ee, et considerons le {\it spectle Birkhoff} associ\'e \`a cette decomposition.
%Nous esquissons une formulae qui relie la dimension Hausdorff des niveau
%aux entropie et exposantes de Lyapunov des mesures invariantes probabilites,
%et ensuite utilisons cette formulae pour etablir la continuite de la spectle.
%Pour obtenir une estimation superieure de la dimension Hausdorff, on a besoin de ``voir" des comportament de beaucoup de points.
%Dans cet optique nous construions une famille des torres.
%En utilisant ces torres nous etablions le principle grande deviasion pour des distributions empiriques, avec Lebesgue comme une mesure de reference.}
\section{Introduction}

Let $X=[-1,1]$, and let $f_a\colon X\circlearrowleft$ be the
quadratic map given by $f_ax=1-ax^2$, where $0< a\leq2$. It is
well-known \cite{BC85,BC91,Jak81} that there exists a set of
$a$-values near $2$ with positive Lebesgue measure for which the
corresponding $f=f_a$ admits an invariant probability 
measure $\mu$ that is absolutely continuous with respect to Lebesgue. 
In this paper we develop a theory of multifractal formalism for a positive measure
set of these quadratic maps.

Given a function $\varphi\colon X\to \mathbb R$ we consider sets of the form
$$K_\varphi(\alpha)=\left\{x\in X\colon\lim_{n\to\infty}
\frac{1}{n}S_n\varphi(x)=\alpha\right\},\ \ \alpha\in\mathbb R,$$
where $S_n\varphi=\sum_{i=0}^{n-1}\varphi\circ f^i$.
The following characteristic of the sets $K_\varphi(\alpha)$ has been studied in the literature: $$
B_\varphi(\alpha)=\dim_HK_\varphi(\alpha),$$ where $\dim_H$ denotes
the Hausdorff dimension. This function of $\alpha$ is called a {\it Birkhoff spectrum} of $\varphi$. In the case $\varphi=\log |Df|$ it is called a {\it Lyapunov spectrum.} 
Multifractal formalism aims to relate these spectra to other characteristics of the system, and to study the regularity of the spectra as functions of $\alpha$, for instance, continuity, smoothness and convexity. With this study one tries to get more refined descriptions of the dynamics
than purely stochastic considerations.

In the creation of the theory of multifractal formalism, 
uniform hyperbolicity or the absence of critical points
have been assumed to obtain good descriptions of the spectra (see e.g. \cite{Cli10, Ols03,Pes97,PesWei97,PesWei01,Wei99}). 
Our aim here is to incorporate into the theory certain nonuniformly expanding quadratic maps
on the interval with critical points.
We provide a simple set of conditions satisfied on a positive measure set in the parameter space of the quadratic maps, and give a partial description of the Birkhoff spectrum when these conditions are met.

We formulate our conditions as follows:
\begin{itemize}
\item[(A1)] $f=f_a$ where $a$ is sufficiently near $2$;
\item[(A2)] $|Df^n(f0)|\geq e^{\lambda n}$ for every $ n\geq0$, where $\lambda=\frac{9}{10}\log2$;
\item[(A3)] $|f^n0|\geq e^{-\frac{1}{100}\sqrt{n}}$ for every $ n\geq1$;
\item[(A4)] $f$ is topologically mixing on $[f^20,f0]$.
\end{itemize}
Benedicks $\&$ Carleson \cite{BC91} proved the 
the abundance of parameters near $2$ for which (A2) holds.
For these parameters, there exists a unique absolutely continuous invariant probability
measure $\mu$ (acip for short).
The abundance
of parameters for which (A3) holds was proved by Benedicks
$\&$ Young \cite{BY92}, and previously by Benedicks $\&$ Carleson
\cite{BC85} under slightly different hypotheses. 
For their parameters, (A4) holds (see \cite[Lemma 2.1]{You92}).
The parameter
sets they constructed have $2$ as a full Lebesgue density point.
Hence, given $a_0<2$ arbitrarily near $2$, there is a set
$A\subset[a_0,2]$ with positive Lebesgue measure such that (A2)-(A4) 
hold for all
$a\in A$.

Let $C(X)$ denote
the space of continuous functions on $X$, and
$\mathcal M_f$ the space of $f$-invariant
probability measures endowed with the topology of weak convergence.
For $\varphi\in
C(X)$ define
$$c_\varphi=\inf_{x\in X}\varliminf_{n\to\infty}\frac{1}{n}S_n\varphi(x)
\ \text{ and } \ d_\varphi=\sup_{x\in
X}\varlimsup_{n\to\infty}\frac{1}{n}S_n\varphi(x).$$
Since $\mathcal M_f$ is compact and $\varphi$ is continuous, one has
$c_\varphi=\min\{\nu(\varphi)\colon\nu\in\mathcal M_f\}$ and
$d_\varphi=\max\{\nu(\varphi)\colon\nu\in\mathcal M_f\}$,
where $\nu(\varphi)=\int\varphi d\nu$.
Define sets
$K_\varphi(\alpha)$ as above,
and consider the decomposition $$X
=\left(\bigcup_{\alpha\in[c_\varphi,d_\varphi]}
K_\varphi(\alpha)\right)\cup \hat K_\varphi.$$ Here, $\hat K_\varphi$ is the set of points in $X$
for which $(1/n)S_n\varphi$ does not converge. 
This decomposition has extremely complicated topological structures.
Indeed, by (A4),
$K_\varphi(\alpha)$ and $\hat K_\varphi$ are dense in $X$
unless they are empty. 
If $c_\varphi<d_\varphi$, then $\hat K_\varphi$
is nonempty and carries the full Hausdorff dimension \cite{BarSch00,Chu10}.
Since $c_\varphi$ and $d_\varphi$ are attained by ergodic measures, both
$K_\varphi(c_\varphi)$ and $K_\varphi(d_\varphi)$ are nonempty.
Using (A4) one can construct points with time averages converging to any number $\alpha\in(c_\varphi,d_\varphi)$. Hence any $K_\varphi(\alpha)$ in the decomposition is nonempty.

Let $h(\nu)$ denote the entropy of $\nu\in\mathcal M_f$
and define 
$\lambda(\nu)=\int\log |Df|d\nu$ which we call the Lyapunov exponent of $\nu$. 
This value is well-defined \cite{BruKel98}, and 
by a result of \cite{NowSan98},
$$\lambda_{\rm inf}=\inf\{\lambda(\mu)\colon \mu\in\mathcal M_f\}>0.$$
Relationships between entropies, Lyapunov exponents, and dimensions of invariant probability 
measures were
studied in the literature \cite{Hof95,Ledr81,You82}.
The next theorem relates the Birkhoff spectrum to entropies and 
Lyapunov exponents of
invariant probability measures.

\begin{theorema}
If $f=f_a$ satisfies {\rm (A1)}-{\rm (A4)}, then for any
$\varphi\in C(X)$ and $\alpha\in[c_\varphi,d_\varphi]$,
$$B_\varphi(\alpha)=\lim_{\varepsilon\to0}{\sup}
 \left\{\frac{h(\nu)}{\lambda(\nu)}\colon \nu\in\mathcal M_f,
\ \left|\nu(\varphi)-\alpha\right|<\varepsilon\right\}.$$
In addition, 
the Birkhoff spectrum $\alpha\mapsto B_\varphi(\alpha)$ is continuous. 
\end{theorema}

%The monotonicity of the spectrum as in Theorem A
%is an immediate consequence of 
 %$B_\varphi(\mu_{\rm ac}(\varphi))=1$ and the next
%\begin{lemma}\label{monotone}
%For all $\alpha_0,\alpha_1\in[c_\varphi,d_\varphi]$ with $\alpha_0<\alpha_1$ and $0\leq t\leq 1$ we have
%$$B_\varphi(t\alpha_0+(1-t)\alpha_1)\geq\min\left(B_\varphi(\alpha_0), B_\varphi(\alpha_1)\right).$$
%\end{lemma}
 %\begin{proof}
%From the formula in Theorem A the following holds for $i=1,2$.
%For each $n>0$ there exists $\varepsilon_{i,n}>0$ such that:
%$\varepsilon_{i,n}\to0$ as $n\to\infty$; there exists $\mu_{i,n}\in\mathcal M_f$ such that
%$B_\varphi(\alpha_i)-1/n\leq h(\mu_{i,n})/\lambda(\mu_{i,n})$
%and $|\mu_{i,n}(\varphi)-\alpha_i|\leq \varepsilon_{i,n}$.
%We have
%$$\min\left(B_\varphi(\alpha_0), B_\varphi(\alpha_1)\right)
%\leq\frac{1}{n}+\min\left(\frac{h(\mu_{0,n})}{\lambda(\mu_{0,n})},
%\frac{h(\mu_{1,n})}{\lambda(\mu_{1,n})}\right).$$
%The minimum of the right-hand-side is
%$\leq
%h(\nu_n)/\lambda(\nu_n)$, where $\nu_n=t\mu_{0,n}+(1-t)\mu_{1,n}$. 
%Then
%\begin{align*}\min\left(\frac{h(\mu_{0,n})}{\lambda(\mu_{0,n})},
%\frac{h(\mu_{1,n})}{\lambda(\mu_{1,n})}\right)\leq{\sup}
 %\left\{\frac{h(\mu)}{\lambda(\mu)}\colon \mu\in\mathcal M_f,
%\ \left|\mu(\varphi)-(t\alpha_0+(1-t)\alpha_1)\right|<t\varepsilon_{0,n}+(1-t)\varepsilon_{1,n}\right\}.\end{align*}
%Combining these two inequalities and then letting $n\to\infty$ we get the desired inequality.
%\end{proof}

In \cite{Chu10}, the first-named author derived the same formula as in Theorem A
for a class of one-dimensional maps. This class includes maps
 whose critical points are non-recurrent and with no neutral or stable 
periodic point (the so-called Misiurewicz maps). 
 Theorem A allows the recurrence of the critical point at a sub-exponential rate
by condition (A3).
Although $\varphi$ is required to be continuous, an extension of the formula
to cover the Lyapunov spectrum will be given in our forthcoming work.

The multifractal formalism for one-dimensional maps with critical points
is a rapidly expanding area of research, and quite a few results have
been obtained lately.
For multimodal maps satisfying growth conditions of derivatives along 
the orbits of critical points,  
Iommi $\&$ Todd \cite{IomTod11} obtained a formula which relates the Lyapunov spectra 
to thermodynamic pressures.
See
Gelfert, Przytycki $\&$ Rams \cite{GelPrzRam10} and 
Przytycki $\&$ Rivera-Letelier  \cite{PrzRiv11}
for results on the Lyapunov spectra of rational maps on the Riemannian sphere.
A key idea common to these recent works
is to construct a sequence of nice induced systems that ``exhausts" the original system. 
Although a proof of Theorem A relies on the same idea, 
 our induced systems are 
 equipped with  a special recurrence property. This requires a new construction.

The formula in Theorem A yields several properties of the Birkhoff spectrum.
For instance, it is easy to show that $B_\varphi$ is monotone increasing
on the interval $[c_\varphi, \mu(\varphi)]$, while it is
monotone decreasing on 
$[ \mu(\varphi),d_\varphi]$ as a function of $\alpha$.
From the formula it readily follows that
$ B_\varphi$ is upper semi-continuous. 
We are able to show that $B_\varphi$ is lower semi-continuous, and so it is continuous.
This phenomenon illustrates what is sometimes called the {\it multifractal miracle} - even though the decomposition of the phase space into the level sets is intricate and extremely complicated, the function $B_\varphi$ which encodes
this decomposition is continuous.

It is an interesting problem to study better regularities of 
the spectrum. If the dynamics is uniformly hyperbolic and the function $\varphi$ 
is H\"older continuous, then the spectrum is real analytic and concave \cite{PesWei01}. 
For one-dimensional maps with parabolic fixed points, 
the non-analyticity of the Lyapunov spectra implies the finiteness of absolutely continuous invariant measures \cite{Nak00}. 
For the quadratic maps, only numerical results are known (see e.g. 
\cite{GraBadPol87,HJKPS86}).
%As for the convexity, it is known that the Lyapunov spectrum is not always concave \cite{IomKiw09}.

Our strategy for the lower estimate of $B_\varphi(\alpha)$ is to construct certain Cantor
sets in $K_\varphi(\alpha)$, and then put probability measures on them for which the {\it Mass Distribution Principle} holds (see \cite[Proposition 2.1]{You82}). %[\cite{Fal03} pp.60]).
The presence of the critical point does not matter 
because small derivatives tend to improve lower estimates of Hausdorff dimension.

For the upper estimate, we  
approximate $B_\varphi(\alpha)$ from above by the dimensions of ergodic measures
(cf. Proposition \ref{upperd}).
To construct such an ergodic sequence we construct
a family of uniformly hyperbolic induced systems with finitely many branches. We pick the corresponding family of equilibrium states for weighted geometric 
potentials, and then spread them out to produce a sequence of ergodic measures with the desired property.

The biggest difficulty is to construct such a family of induced systems.
We do this in two steps. We first construct a family of towers, with a special property that
a {\it  positive definite fraction of points in each
partition element quickly fall down to the
ground floor}. 
For this construction 
we make an important use of condition (A3).
We then construct the desired family of induced systems by choosing
a subsystem from each tower dynamical system.

%In very special cases such as the case of Misiurewicz maps, l
%Towers (or inducing schemes) are usual tools for the study of dynamical systems lacking the uniform hyperbolicity. 
%There is a much freedom and flexibility in the construction of towers, and so the issue is to construct a ``nice tower" which captures a relevant information of the underlying system. 
%In \cite{Chu10}, a formula similar to the one in Theorem A was obtained for one-dimensional maps admitting ``ideal towers". However, it is hard to construct such towers apart from very special cases (e.g. the case of Misiurewicz maps).
%When it is hard to get a necessary information just by considering a single tower, a natural solution is to construct a family of towers and use them altogether. We refer the reader to e.g. \cite{BruTod09,ChuTak,IomTod11,PrzRiv11} in which this type of solution was taken, for different purposes with different assumptions on the systems.

 Using the family of towers used in the proof of Theorem A we establish a large 
deviation principle for the Lebesgue measure.
Let $\mathcal M$ denote
 the space of probability measures on $X$ endowed with the topology of weak convergence.
Define a free energy
function $F\colon \mathcal M\to\mathbb R\cup\{-\infty\}$ by
\[F(\nu)=\begin{cases} &h(\nu)-\lambda(\nu)\ \
\text{if}\ \ \nu\in\mathcal M_f;  \\
&-\infty \ \ \ \ \ \ \ \ \ \ \ \ \text{otherwise.}\end{cases}\] By
Ruelle's inequality \cite{Rue78}, $F(\nu)\leq0$ and the equality
holds only if $\nu=\mu$  \cite{Led81}.

It is known \cite{BruKel98} that the Lyapunov exponent is not lower semi-continuous,
and so $-F$ may not be lower semi-continuous.
Hence we introduce its lower-semi-continuous regularization $I\colon\mathcal M\to[0,\infty]$
by $$I(\nu)= - \inf_{\mathcal G} \sup \{ F(\xi)\colon \xi\in \mathcal G\},$$
where the infimum is taken over all neighborhoods $\mathcal
G$ of $\nu$ in $\mathcal M$.
Denote by $|$ $\cdot$ $|$ the Lebesgue measure on $X$ and let
$\log0=-\infty$.
Let $\delta_x^n=(1/n)\sum_{i=0}^{n-1}\delta_{f^ix}$
where  $\delta_{f^ix}$ is the
Dirac measure at $f^ix$.

\begin{theoremb}
Let $f=f_a$ satisfy {\rm (A1)}-{\rm (A4)}. Then the large
deviation principle holds for $(f, | \cdot |)$ with $I$ the rate function, namely
for any open set $\mathcal G\subset\mathcal
M$, 
\begin{equation*}\label{low}
\varliminf_{n\to\infty}\frac{1}{n} \log |\{x\in X\colon
\delta_x^n\in \mathcal G\}|\geq-\inf\{I(\nu)\colon {\mu\in \mathcal
G}\},\end{equation*}
and  for any closed set $\mathcal K\subset\mathcal
M$,
\begin{equation*}\label{up}
\varlimsup_{n\to\infty}\frac{1}{n} \log |\{x\in X\colon
\delta_x^n\in \mathcal K \}|\leq-\inf\{I(\nu)\colon\mu\in \mathcal
K\}.\end{equation*}
\end{theoremb}
The large deviation principle has been proved in different settings, for
different reference measures and with different assumptions on the hyperbolicity
of the systems \cite{ComRiv11,K,OP,PrzRiv11}.
For a positive measure set of quadratic maps we treat here, 
the large deviation principle for the acips was proved in \cite{ChuTak}. 
Theorem B is not a consequence of this, 
because the density of the acip is unbounded.

The Contraction Principle in large deviations \cite{DZ} allows us to obtain a formula for fluctuations
of time averages of continuous functions. 
Let $\varphi\in
C(X)$. We assume $c_\varphi<d_\varphi$, for otherwise it is meaningless
to consider $\varphi$. Define a function $F_\varphi\colon [c_\varphi,d_\varphi]\to\mathbb R$ by
$$F_\varphi(\alpha)=\sup\left\{F(\nu)\colon \nu\in\mathcal M_f,\ \nu(\varphi)
=\alpha\right\}.$$

\begin{cor}\label{contract}
If $c_\varphi\leq\alpha<\beta\leq d_\varphi$, then
\[\lim_{n\to\infty}\frac{1}{n}\log \left|\left\{
\alpha\leq \frac{1}{n}S_n\varphi\leq \beta\right\}\right|=\max_{\alpha\leq t\leq \beta
}F_\varphi(t).\]
\end{cor}
Keller $\&$ Nowicki \cite{KelNow92} obtained a {\it local} result
which claims the existence of the limit provided $\varphi$ is H\"older continuous
and $\alpha$, $\beta$ are sufficiently near the mean $\mu(\varphi)$.
Corollary 1 is a {\it full} result with no restriction 
on $\alpha$ or $\beta$.

The next corollary follows from Varadhan's integral lemma
\cite[p.137]{DZ} and the convex duality of Fenchel-Legendre
transforms \cite[p.152]{DZ}.
\begin{cor}
For any $\varphi\in C(X)$, the limit
\[P(\varphi)=\lim_{n\to\infty}\frac{1}{n}\log \int e^{S_n\varphi}dx\]
exists. In addition, $(P,I)$ form a Legendre pair, namely the following holds:
$$P(\varphi)=\max\left\{
\nu(\varphi)-I(\nu)\colon\nu\in\mathcal M_f\right\}\ \ \text{for
all} \ \varphi\in C(X);$$ 
$$I(\nu)=\max\left\{\nu(\varphi)-P(\varphi)\colon\varphi\in C(X)\right\}\ \
\text{for all}\ \ \nu\in\mathcal M_f.$$
\end{cor}

The rest of this paper consists of four sections. In Sect.2 we
construct a family of towers, collecting 
materials in \cite{ChuTak} as far as needed. 
In Sect.3, using this family of towers we estimate  
$B_\varphi(\alpha)$ from above. In Sect.4 we estimate  
$B_\varphi(\alpha)$ from below and complete the 
proof of the formula in Theorem A. 
We then use this formula to prove the continuity of the Birkhoff spectrum.
In Sect.5 we prove Theorem B.

\section{Construction of a family of towers}
In this section, for a map $f$ satisfying (A1)-(A4) 
we first introduce the machinery in \cite{ChuTak} for recovering small derivatives near 
the critical point. We then 
construct a family of induced maps and associated towers. 
Important constants are
$0<\varepsilon\ll1$ and $N\gg1$, chosen in this order. 
In this section we suppose they are given.
In Sect.3 and Sect.5 we let $\varepsilon$ approach $0$.

We use the following standard notation:
for a set $A\subset X$,
 $d(0,A)=\inf\{|x|\colon x\in
A\}$; 
given a partition $\mathcal P$ of $A\subset X$
and $B\subset A$,
$\mathcal P|B=\{\omega\cap B\colon \omega\in\mathcal P\}$.
%We use both $m($ $\cdot$ $)$ and 

\subsection{Recovering expansion}\label{recovery}

 The next lemma
states that the dynamics outside of a small neighborhood of 
the critical point is
uniformly expanding with an exponent independent
of the size of the neighborhood. 
%It will be used to construct induced maps with arbitrarily small scale.

\begin{lemma}\label{exp2}{\rm (cf. \cite[ Lemma 2.5]{ChuTak})}
The following holds for any $\hat\delta>0$:
if $x\in X$, $n\geq1$ are such that $|f^ix|\geq\hat\delta$ for
every $0\leq i\leq n-1$, then $|Df^{n}(x)|\geq \hat\delta
e^{\frac{\lambda}{3} n}$. Moreover, if $|f^nx|<\hat\delta$, then
$|Df^{n}(x)|\geq e^{\frac{\lambda}{3} n}$.
\end{lemma}
A proof of this lemma is 
almost identical to that of \cite[ Lemma 2.5]{ChuTak}, and hence is omitted. 
Although particular values of $\hat\delta$ were chosen there,
this choice is not essential.

 To deal with the loss of expansion due
to returns to the critical region we mimic the binding
argument of Benedicks $\&$ Carleson \cite{BC85,BC91}: subdivide the
interval into pieces, and deal with them independently. 
For $p>0$ let
\begin{equation}\label{Theta}
\delta_p=\sqrt{\frac{e^{-\varepsilon p}}{10}\left[\sum_{i=0}^{p-1}\frac{|Df^i(f0)|}{|f^{i+1}0|}\right]^{-1}}.
\end{equation}
If $\delta_p\leq |x|<\delta_{p-1}$, then
we regard the orbit of $x$ as bound to the orbit of $0$ up to time
$p$.

\begin{lemma}\label{P}
For any $\varepsilon>0$ there exists $N>0$ such that if $p\geq N$ and
 $\delta_p\leq |x|<\delta_{p-1}$, then:
\begin{itemize}

\item[(a)] $|Df^{p}(x)|\geq
e^{\frac{\lambda}{3}p}$;

\item[(b)]  $\log|x|^{-\frac{2}{\log 5}}\leq 
p\leq\log|x|^{-\frac{2}{\lambda}}$.
\end{itemize}
\end{lemma}
\begin{proof}
(a) and the second inequality in (b) are due to \cite[Lemma 2.3]{ChuTak}.
Rearranging $|x|^2\geq \delta_p^2\geq 5^{-p}$ which follows from
the definition of $\delta_p$ in \eqref{Theta}, and then using $|Df|\leq 4$, 
(A3) yield the first inequality in (b).
\end{proof}

\subsection{Construction of a partition with slowly recurrent points.}\label{slow}
We construct a partition of a small neighborhood of the critical point which is well-adapted
to later constructions.

To start,
for each $p> N$ cut the interval $[\delta_p,\delta_{p-1})$ into $\left[ e^{3\varepsilon
p}\right]$-number of intervals of equal length and 
denote them by $\hat I_{p,j}$ $(j=1,2,\ldots,\left[ e^{3\varepsilon
p}\right])$, from the right to the left. %Let
%$\hat I_{p,-j}=-\hat I_{p,j}$, which is the mirror image of $\hat I_{p,j}$ with respect to $0$.
This defines a partition of the interval $(0,\delta_N)$, but it is not
satisfactory for our construction, because
there is no control over the iterates of the boundary points of the partition elements. 
To rectify this, we show in the next lemma the existence of a point
in each $\hat I_{p,j}$ which is slowly recurrent to the critical point $0$.
We then use these points as partition points.

 \begin{lemma}\label{miwa}
For each $(p,j)$ there exists 
$x\in \hat I_{p,j}$ such that $|f^{n}x|\geq
\delta_N e^{-\varepsilon n}$ for every $n\geq
\varepsilon^{-1}$.
\end{lemma}

\begin{proof}
Set $t_0=0$, $\omega_0=\hat I_{p,j}$ and $p_0=p$. 
For every $\varepsilon^{-1}\leq n\leq p_0$ we have
\begin{equation}\label{meq3}d(0,f^n\omega_{0})\geq |f^n0|-|f^n\omega_0|
\geq(1/2)|f^n0|\geq(1/2)e^{-\frac{\sqrt{n}}{100}}\geq e^{-\varepsilon n},\end{equation}
where we have used the bounded distortion of $f^{n-1}$ on $f\omega_0$ from \cite[Lemma 2.1]{ChuTak} for the second inequality. The third one follows from (A3).
The last one holds for sufficiently small $\varepsilon$.

By induction we choose 
a sequence $n_0<n_1<\cdots$ of integers and a sequence
$\omega_0\supsetneq\omega_1\supsetneq\cdots$ 
of closed intervals such that 
for every $k\geq0$,
\begin{equation}\label{induction}
 f^{n_k}\omega_k=\hat I_{p_k,j_k}\ \text{for some $p_k,j_k$ and}\ d(0,f^{n}\omega_{k})\geq \delta_N
e^{-\varepsilon n}\ \ \text{for every $n_k\leq n\leq n_{k}+p_k-1$}.\end{equation} 
From \eqref{meq3} and \eqref{induction},
the point in the singleton $\bigcap_{k\geq0}\omega_k$ satisfies the desired property.

For the rest of the proof, we assume \eqref{induction} holds for some $k=l$,
and then indicate how to choose $t_{l+1}$ and $\omega_{l+1}$ 
for which \eqref{induction} holds for $k=l+1$.
An argument to show \eqref{induction} for $k=0$ is included in
the general step of the induction below.

Given $n_l$, $\omega_l$
such that  $f^{n_l}\omega_l=\hat I_{p_l,j_l}$,
 define 
$n_l+p_l\leq t_1<t_2<\cdots$ inductively as follows:
 $t_1$ is the smallest $t\geq n_{l}+p_{l}$ with $d(0,f^t\omega_l)<\delta_N$.
Suppose $t_i$ has been defined. If $f^{t_i}\omega_l$ intersects
no more than two $\hat I_{p,j}$-intervals,
say $\hat I_{p,j}$ and $\hat I_{p',j'}$, $p\leq p'$ (possibly $\hat I_{p,j}=\hat I_{p',j'}$),
then define $t_{i+1}$ to be the smallest $t\geq t_{i}+q_i$ with
$d(0,f^t\omega_l)<\delta_N$,
where $q_i=p$ if $f^{t_i}\omega_l\subset(-\delta_N,\delta_N)$, and $q_i=1$ 
otherwise.
If $f^{t_{i}}\omega_l$ intersects
more than three $\hat I_{p,j}$-intervals, then $t_{i+1}$ is undefined.

The expansion
estimates in Lemma \ref{exp2} and Lemma \ref{P}
imply that one finally reaches $t_s$ such that 
$f^{t_{s}}\omega_l$ intersects
more than three $\hat I_{p,j}$-intervals.
For all $\theta\in f^{n_l+p_l}\omega_l$ we have
\begin{equation}\label{ineq}
|Df^{n_{l+1}-n_l-p_l}(\theta)|\geq \delta
\exp\left(\frac{\lambda}{3}\sum_{\stackrel{1\leq i\leq s-1}{q_i\neq1}}q_i\right)\geq\delta_N.\end{equation}
This and  $|f^{n_l+p_l}\omega_l|\geq e^{-5\varepsilon p_l}$ which follows from
\cite[Lemma 2.6(a)]{ChuTak} yield
$|f^{n_{l+1}}\omega_l|\geq \delta_N e^{-5\varepsilon p_l}.$ By Lemma
\ref{P}(b),
$p_l
\leq (2/\lambda)(-\log\delta_N+\varepsilon n_l)\leq
(3/\lambda)\varepsilon n_l,$ and thus $|f^{n_{l+1}}\omega_l|\geq \delta_N
e^{-5\varepsilon p_l} \geq\delta_N
e^{-\frac{15}{\lambda}\varepsilon^2 n_{l+1}}$. From this
and the upper estimate
of the length of $\hat I_{p,j}$ in 
\cite[Lemma 2.6(b)]{ChuTak}, one can choose $\omega_{l+1}\subset \omega_l$ 
such that $f^{n_{l+1}}\omega_{l+1}=\hat I_{p_{l+1},j_{l+1}}$
and $d(0, f^{n_{l+1}}\omega_{l+1})\geq\delta_N e^{-\varepsilon n_{l+1}}$.

It is left to estimate the distance of the forward 
iterates of $f^{n_l+p_l-1}\omega_{l+1}$ to the critical point.
We first consider the case $n=t_i$ with $q_i\neq1$.
\eqref{ineq} implies that for some $\theta\in
f^{n_l+p_l}\omega_l$ we have
\[
2\geq |f^{n_{l+1}}\omega_l|=|Df^{n_{l+1}-n_l-p_l}(\theta)|
\cdot |f^{n_l+p_l}\omega_l|\geq \delta_N
e^{\frac{\lambda}{3}q_{i}}
e^{-5\varepsilon p_l}.\] Taking logs and then rearranging the
result we have
\begin{equation}\label{inessential}
q_{i}\leq
-(4/\lambda)\log\delta_N+(15/\lambda)\varepsilon
p_l\leq(16/\lambda)\varepsilon
p_l\leq(48/\lambda^2)\varepsilon^2n_l.
\end{equation}
Hence, for $n=t_i$ we obtain
\begin{equation*}
d(0,f^{n}\omega_l)\geq\delta_{q_{i}}\geq 5^{-q_{i}} \geq
e^{-\varepsilon n_l}\geq \delta_N e^{-\varepsilon n}.\end{equation*} 

Next we consider the case
$n\in(t_{i},t_{i}+q_{i}]$ with $q_i\neq1$.
Let $J$ denote the minimal interval containing
$f^{t_{i}}\omega_l$ and $0$.
Then
$|fJ|\leq 2\delta_{q_{i}}^2,$ and 
the distortion of
$f^{n-t_{i}-1}$ on $f J$ is bounded by
\cite[Lemma 2.1]{ChuTak}. 
Hence
$$|f^{n-t_i}J|\leq2|Df^{n-t_{i}-1}(f0)| |fJ|\leq
4|Df^{n-t_{i}-1}(f0)|\delta_{q_{i}}^2\leq(9/10)|f^{n-t_{i}+1}0|,$$
and thus
\begin{equation*}
 d(0,f^n\omega_l)\geq |f^{n-t_i+1}0|-|f^{n-t_i}J|\geq(1/10)|f^{n-t_{i}+1}0|
\geq(1/10)e^{-\frac{1}{100}\sqrt{q_{i}}} \geq e^{-\varepsilon t_{i}}
\geq\delta e^{-\varepsilon n}.\end{equation*}
 The second inequality follows from (A3) and
the third from \eqref{inessential}.

Note that the above argument may be extended to the case 
$n\in[n_{l+1}+1,n_{l+1}+p_{l+1}-1]$.
For all $n\in [n_l+p_l-1,n_{l+1}-1]$
other than those treated so far, the desired estimate holds
because $f^n\omega_l$ is not
contained in $(-\delta_N,\delta_N)$ and intersects at most one $\hat I_{p,j}$.
The assumption of the induction has been recovered. \end{proof}

%All the estimates in
%\cite{ChuTak} relative to $I_{p,j}$ continue to hold, with
%$I_{p,j}$ replaced by $I_{p,j,{\rm new}}$ and with slightly worse
%constants. 

%\begin{lemma}\label{holder}
%The following estimates hold:
%\begin{itemize}
%\item[(a)] $|f^pI_{p,j}|\geq e^{-5\varepsilon p}$;
%\item[(b)] $|I_{p,j}|\leq d(0,I_{p,j})^{1+\frac{\varepsilon}{3}}$;
%\item[(c)] for all $x,y\in I_{p,j}$,
%$\log\frac{J^p(x)}{J^p(y)}\leq |f^px-f^py|^{\varepsilon^2}.$
%\end{itemize}
%\end{lemma}

%We
%finish the construction of the adapted partition. 
In view of Lemma \ref{miwa}, for each  $\hat I_{p,j}$ fix once and for all a point
$x_{p,j}\in \hat I_{p,j}$ such that $|f^{n}x_{p,j}|\geq
\delta_N e^{-\varepsilon n}$ holds for every $n\geq
\varepsilon^{-1}$. 
Set $\delta=x_{N,1}.$
Note that $\delta<\delta_N$, and $\delta\to0$ as $N\to\infty$.
Using the points $x_{p,j}$ as partition points
we construct a countable partition of the interval $(0,\delta)$ 
in such a way that:\footnote{Note
that not all  $x_{p,j}$ are used in this construction, because of the requirement (ii).}

\begin{itemize}

\item[(i)]
each element of the partition contains exactly one element of $\{\hat I_{p,j}\}$;

\item[(ii)]
each element of the partition is contained in three contiguous elements of
$\{\hat I_{p,j}\}$.
\end{itemize}
The construction is straightforward. The boundary points of 
partition elements belong to $\{x_{p,j}\}.$ 
Let $I_{p,j}$ denote the element of the partition containing
$\hat I_{p,j}$.
Let $ I_{p,-j}=-I_{p,j}$, the mirror image of $I_{p,j}$ with respect to $0$.

\subsection{Construction of dynamical partitions}\label{subdivide}
Let $\Lambda^+=I_{N,1}$, which is the right extremal $I_{p,j}$-interval.
Let
$\Lambda^-=-\Lambda^+$ and $\Lambda=\Lambda^-\cup\Lambda^+.$ 
Let $\hat x$ denote the orientation reversing fixed point of $f$ in $X$ and
set $\hat X=[-\hat x,\hat x]$. 
By induction on the number of iterations we construct a ``decreasing" sequence
$\{\tilde{\mathcal P}_n\}_{n\geq 0}$ of partitions of
$\hat X$ into intervals, and introduce the notion of bound/free states.

Start
with $\tilde{\mathcal P}_0=\{[-\hat x,-\delta],[\delta,\hat x]\}\cup \{I_{p,j}\}_{p,j}$. 
We refer to the intervals 
$f[-\hat x,-\delta]$, $f[\delta,\hat x]$, $fI_{p,j}$ and $f^pI_{p,j}$ as {\it free}, and to $f^iI_{p,j}$
$(1\leq i\leq p-1)$ as {\it bound}.
Call $p$ a {\it bound period of $I_{p,j}$ at time $0$}.

Let $n\geq 1$.
The $f^n$-images of elements of $\tilde{\mathcal
P}_{n-1}$ are in two phases: either \emph{bound} or \emph{free}.
If $\omega\in\tilde{\mathcal P}_{n-1}$, $f^n\omega$ is free and $d(0,f^n\omega)<\delta$,
then $\tilde{\mathcal P}_n$ subdivides $\omega$. For each 
resulting element 
$\omega'\in\tilde{\mathcal P}_n|\omega$ with $d(0,f^n\omega')<\delta$
an integer $p_n(\omega')$ is attached; this integer is 
called a \emph{bound period of $\omega'$ at time $n$}. 
%We say $n$ is a {\it free return time} of $\omega'$.

Given $\omega\in\tilde{\mathcal P}_{n-1},$ $\tilde{\mathcal
P}_{n}|\omega$ is defined as follows. 
If
$f^n\omega$ is free and
contains more than two $I_{p,j}$-intervals, then let $\tilde{\mathcal P}_n$
subdivide $\omega$ according to the $(p,j)$-locations of its
$f^n$-image. 
In all other cases, let $\tilde{\mathcal
P}_{n}|\omega=\{\omega\}$.
Partition points are inserted only to ensure that
the $f^n$-images of $\tilde{\mathcal P}_n$-elements intersecting $(-\delta,\delta)$ contain
exactly one $I_{p,j}$.
The $f^n$-images out of $(-\delta,\delta)$ are treated as follows.
Let $\omega'\subset\omega$ be such that $f^n\omega'$ is a component of $f^n\omega\setminus
(-\delta,\delta)$. We let $\omega'\in\tilde{\mathcal P}_n$ if $|f^n\omega'|\geq| \Lambda^+|$.
Otherwise, we glue $\omega'$ to the adjacent element whose $f^n$-image is
contained in $\Lambda^\pm$.

The bound periods at time $n$ of the elements of
$\tilde{\mathcal P}_n|\omega$
 are determined by the $p$-locations of their $f^n$-images. 
 Namely, if $\tilde{\mathcal P}_n$ subdivides $\omega$ and
 $\omega'\in\tilde{\mathcal P}_n|\omega$, then
 $p_n(\omega')=p$ where $p$ is such that $f^n\omega'\supset I_{p,j}$ holds for some $j$.
 If $\tilde{\mathcal P}_n|\omega=\{\omega\}$, then
 $p_n(\omega)=\min\{p\colon I_{p,j}\cap f^n\omega\neq\emptyset\ \text{ for some $j$}\}$.

To proceed, for $\omega'\in\tilde{\mathcal P}_n$
we say $f^{n+1}\omega'$ is {\it bound} if
there exists $k\leq n$ such that $\omega'\in\tilde{\mathcal P}_k$,
$p_k(\omega')$ makes sense and satisfies
$n+1<k+p_k(\omega')$. Otherwise we say $f^{n+1}\omega'$ is {\it
free}. This completes the construction of $\tilde{\mathcal P}_n$ $(n=0,1,\ldots)$.

The following bounded distortion can be proved similarly\footnote{Although the value of 
$``\delta"$ is slightly different from the one used in \cite{ChuTak}, the technical adjustment is minimal.
} 
to \cite[Lemma 2.7]{ChuTak}.
Set $C_0=\exp(-\delta^3)$. Let
$\omega\in\tilde{\mathcal P}_{n-1}$ and suppose that  
$f^n\omega$ is free. Then 
\begin{equation}\label{bounddist}
\frac{|Df^n(x)|}{|Df^n(y)|}\leq C_0 \ \ \forall x,y\in\omega.
\end{equation}

\subsection{Inducing time estimates}
We define inductively a partition $\mathcal
Q$ of $\Lambda$ into intervals 
and an associated inducing time $R\colon\mathcal Q\to\mathbb N$ as follows.
Let $\omega\in\tilde{\mathcal P}_{n-1}|\Lambda$. If
$f^{n}\omega$ is free and
 $f^{n}\omega\supset3\Lambda^+$ or $3\Lambda^-$, then
set $\omega\cap f^{-n}\Lambda^{+}\in\mathcal Q$ or $\omega\cap
f^{-n}\Lambda^{-}\in\mathcal Q,$ and $R(\omega)=n$. We iterate the
remaining parts $f^{n}\omega\setminus\Lambda^+$ or
$f^{n}\omega\setminus\Lambda^-$, which is the union of elements of
$\tilde{\mathcal P}_{n}$, and repeat the same procedure.
By definition, 
for each $\omega\in\mathcal Q$, $f^{R(\omega)}$ sends $\omega$
diffeomorphically onto $\Lambda^+$ or $\Lambda^-$.

Set
$\hat\zeta=|\Lambda^+|/(2C_0)\in(0,1)$ and $C_1=1+\hat\zeta^{-1}.$
Note that $\hat\zeta\to0$ and $C_1\to\infty$ as $\delta\to0$.
Let $\theta=10^{-10000}$ and set $\zeta=\max\{e^{-\frac{\lambda}{14}},(1-\hat\zeta)^{ \theta }\}$.
%The choice of $\theta$ will be made explicit in the proof.
Lemma \ref{escape} below 
applied to $\Lambda^\pm$ implies that 
the measure of the tail set
$$\{R>n\}=\bigcup_{\omega\in\mathcal Q\colon R(\omega)>n}\omega$$
decays exponentially fast. In particular,
$\mathcal Q$ is a partition of a full measure subset of $\Lambda$.
\begin{lemma}\label{escape}
There exists $k_0=k_0(\delta)$ such that the following holds for every $k\geq k_0$:
let $\omega\in\tilde{\mathcal P}_{k-1}$ and suppose that 
$\omega\subset \{R>k\}$ and
$f^k\omega$ is 
free. Then 
$$|\{R> k+l\}\cap\omega|\leq  C_1\zeta^l|\omega|\ \text{ for every } l\geq (16\varepsilon/\lambda) k.$$  
\end{lemma}
\begin{proof}
%Essentially the same as the proof of \cite[Lemma 2.8]{ChuTak}. 
 Let 
$\mathcal Q'$ denote the set of all 
$\omega'\in \mathcal Q| \{R>k+l\}\cap\omega$
 for which there exists $ n\in [k,k+l]$ such that $d(0,f^n\omega_n)
<\delta$ holds for the element $\omega_n\in\tilde{\mathcal P}_n$ containing $\omega'$. Let $\mathcal Q''$ denote the collection of elements
of $\mathcal Q| \{R>k+l\}\cap\omega$ which do not belong to 
$\mathcal Q'$.

Each $\eta\in\mathcal Q'$ 
has an
\emph{itinerary} $(n_1,p_1,j_1),
\ldots,(n_s,p_s,j_s)$ $(s\leq [l/N])$ that is defined as follows:
$k\leq n_1<\cdots<n_s\leq k+l$ is a sequence of integers,
associated with a sequence $\omega_0\supset\omega_{n_1}\supset\cdots\supset
\omega_{n_s}\supset\eta$ of intervals such that $\omega_{n_i}$ $(i=1,\ldots,s)$ is
the element of $\tilde{\mathcal P}_{n_i}$ containing $\omega$ that arises out of the subdivision at time $n_i$, with $I_{p_i,j_i}\subset f^{n_i} \omega_{n_i}$;
$\omega_{n_s}\in\tilde{\mathcal P}_{k+l-1}$.
Since $|f^{n_s+p_s}\omega_{n_s}|\leq2$, 
for some $x\in\omega_{n_s}$ we have
$|\omega_{n_s}|\leq 2|Df^{n_s+p_s}(x)|^{-1}$.
By Lemma \ref{exp2} and Lemma \ref{P}(a) 
we have
$|Df^{n_{s}+p_s}(x)|\geq \delta e^{\frac{\lambda}{3}\sum_{i=1}^sp_i}|Df^k(x)|$,
and by \eqref{bounddist},
$|Df^k(x)|\geq (1/C_0)|f^k\omega|/|\omega|$. 
Combining these three inequalities we obtain
$$|\eta|\leq|\omega_{n_s}|\leq 2C_0\delta^{-1} e^{-\frac{\lambda}{3}\sum_{i=1}^sp_i}
\frac{|\omega|}{|f^k\omega|},$$
and therefore
\begin{align*}\label{ret}\sum_{\eta\in\mathcal Q'}|\eta|&\leq 2C_0\delta^{-1}\sum_s\sum_{P}e^{-\frac{\lambda}{3}P}
\#\left\{\{(n_i,p_i,j_i)\}_{i=1}^s\colon \sum_{i=1}^sp_i=P\right\}
\frac{|\omega|}{|f^k\omega|}.\end{align*}
From the proof of \cite[Lemma 2.8]{ChuTak}
the cardinality is $\leq e^{\varepsilon n}e^{4\varepsilon P},$
and from the proof of \cite[Sublemma 2.9]{ChuTak}
$n_{i+1}-n_i\leq 2p_i$ holds
for every $1\leq
i\leq s$, where $n_{s+1}> k+l$ is such that $\tilde{\mathcal
P}_{n_{s+1}}$ partitions $\omega_{n_s}$.
It follows that
$l<n_{s+1}\leq n_1+2\sum_{i=1}^{s}p_i.$ 
If $n_1\leq k+l/2$ then
$\sum_{i=1}^{s}p_i\geq l/4$, and therefore 
\begin{equation}\label{ret1}\sum_{\stackrel{\eta\in\mathcal Q'}{n_1\leq k+l/2}}|\eta|
\leq C_0\delta^{-1}\frac{l}{N}\sum_{P\geq l/4}
e^{(8\varepsilon-\frac{\lambda}{3})P}\frac{|\omega|}{|f^k\omega|}\leq
e^{-\frac{\lambda}{13}l}\frac{|\omega|}{|f^k\omega|}.\end{equation}
where the last inequality holds provided $k$ is sufficiently large 
because $l\geq \sqrt{\varepsilon} k$.

For those $\eta\in\mathcal Q'$ with $n_1>k+l/2$, 
a similar reasoning shows
$$|\eta|\leq|\omega_{n_1}|\leq C_0\delta^{-1} e^{-\lambda(n_1-k)}
\frac{|\omega|}{|f^k\omega|}\leq  C_0\delta^{-1} e^{-\frac{\lambda l}{2}}
\frac{|\omega|}{|f^k\omega|}
\leq  e^{-\frac{\lambda l}{3}}
\frac{|\omega|}{|f^k\omega|},$$
and therefore
\begin{equation}\label{ret2}\sum_{\stackrel{\eta\in\mathcal Q'}{n_1> k+l/2}}|\eta|
\leq C_0\delta^{-1}
\frac{l}{N}\sum_{P\leq l}
e^{4\varepsilon P+\varepsilon l-\frac{\lambda}{3}l}\frac{|\omega|}{|f^k\omega|}\leq
e^{-\frac{\lambda l}{4}}\frac{|\omega|}{|f^k\omega|}.\end{equation}

 We now treat elements of $\mathcal Q''$.
 Let $t_1\geq k$ be such that $\omega$ is subdivided at time $t_1$.
 Since $I_{p,j}\supset f^n\omega$ holds for some $n<k$
 we have $|f^k\omega|\geq\delta e^{-5\varepsilon k}$.
If $t_1-k\geq 16(\varepsilon/\lambda) k$, then 
$|f^{t_1}\omega|\geq
\delta e^{-5\varepsilon k}e^{\frac{\lambda}{3}(t_1-k)}\geq \delta e^{\varepsilon k}>2=|X|,$
which is a contradiction. 
 Hence $t_1-k<(16\varepsilon/\lambda) k$, and so $t_1<k+l$.

Let $t\geq k$. We say
$\tilde\omega\in\tilde{\mathcal P}_t|\omega$ is an {\it escaping component at time $t$}
if $\tilde\omega$ arises out of subdivision at time $t$ and satisfies
$d(0,f^t\tilde\omega)=\delta$.
 Let $\mathcal E_1$ denote the collection of escaping components at time $k+r$.
 If $\mathcal E_1=\emptyset$, then $\mathcal Q''=\emptyset$.
 Hence we assume $\mathcal E_1\neq\emptyset$. 

Each $\eta\in\mathcal Q''$ 
has an
\emph{itinerary} $(t_1,\epsilon_1),
\ldots,(t_q,\epsilon_q)$ that is defined as follows:
$k\leq t_1<\cdots<t_q< k+l$ is a sequence of integers,
associated with a nested sequence $\omega\supset\omega_{t_1}\supset\cdots\supset
\omega_{t_q}\supset\eta$ of intervals such that for each $i$, $\omega_{t_i}$ is
an escaping component at time $t_i$ and $\epsilon_i=+$ (resp. $\varepsilon_i=-$)
if $f^{t_i} \omega_{t_i}$ is at the right (resp. left) of the critical point;
$\omega_{t_q}$ is the smallest escaping component containing $\eta$.
Call $q$ the {\it length of the itinerary of} $\eta$.
%Using the previous estimates and the fact that $\omega_{t_q}$ is not subdivided up to time
%$k+l-1$, we have 
%$|\omega_{t_q}|\leq e^{-\lambda l/2}|\omega|/|f^k\omega|.$

For $\theta>0$ let
 $\mathcal Q''_{\leq\theta l}=\{\eta\in\mathcal Q''\colon\text{The length of the itinerary is $\leq \theta
l$}\}$.
The number of all itineraries of length $q$ is $\leq2^q\left(\begin{smallmatrix}l\\q\end{smallmatrix}\right)$, and so the Stirling formula implies one can choose
small $\theta$ such that $\#\mathcal Q''_{\leq\theta l}\leq e^{\lambda l/100}$.
Then
\begin{equation}\label{ret3}\sum_{\eta\in
\mathcal Q''_{\leq\theta l}}|\eta|\leq\#\mathcal Q''_{\theta l}e^{-\lambda l/2}\frac{|\omega|}{
|f^k\omega| }\leq
e^{-\frac{\lambda}{3}l}\frac{|\omega|}{|f^k\omega|}.\end{equation}

To treat elements in 
$\mathcal Q''_{>\theta l}=\{\eta\in\mathcal Q''\colon\text{The length of the itinerary is $> \theta
l$}\}$, for each $q\geq1$
 define a collection $\mathcal E_q$ of escaping components (at variable times) inductively as follows: 
  each $\omega\in\mathcal E_q$ is an escaping component at 
some time, say $t=t(\omega)$.
Let $t'>t$ denote the time at which $\omega$ is subdivided. 
Then $\omega$ contains no or at most two escaping components at time $t'$. 
We let them in $\mathcal E_{q+1}$.
Let $E_q=\bigcup_{\omega\in\mathcal E_q}\omega$.
If $\omega\in\mathcal E_q$ and $\omega\cap E_{q+1}\neq\emptyset$, then
from the bounded distortion \eqref{bounddist},
$$\frac{|\omega\setminus E_{q+1}|}{|\omega|}\geq C_0^{-1}\frac{|f^{t'}(\omega\setminus E_{q+1})|}{|f^{t'}\omega|} \geq C_0^{-1}\frac{|\Lambda^+|}{|X|}=\hat\zeta,$$
and so
$|\omega\cap E_{q+1}|\leq|\omega|-|\omega\setminus E_{q+1}|\leq(1-\hat\zeta)|\omega|.$
 Hence
 $|E_{q+1}|\leq(1-\hat\zeta)|E_q|,$ and thus
$|E_q|\leq(1-\hat\zeta)^q|\omega|$.
By definition, if the itinerary of $\eta\in\mathcal Q''_{>\theta l}$ is of length $q$,
then $\eta$ is contained in an element of $\mathcal E_q$.
Hence
 \begin{equation}\label{ret4}\sum_{\eta\in
\mathcal Q''_{>\theta l}}|\eta|\leq\sum_{\theta l\leq q\leq l}|E_q|\leq
\sum_{q\geq\theta l}(1-\hat\zeta)^q\leq \hat\zeta^{-1}(1-\hat\zeta)^{ \theta l}.\end{equation}

Since $\theta$ is independent of $\delta$ and $\hat\zeta\to0$ as $\delta\to0$,
we have $e^{-\frac{\lambda}{14}}\leq\zeta$.
\eqref{ret1} \eqref{ret2} \eqref{ret3} \eqref{ret4} yield 
 $ |\{R>k+l\}\cap\omega|
\leq e^{-\frac{\lambda}{14}l}+\hat\zeta^{-1}(1-\hat\zeta)^{ \theta l}\leq C_1\zeta^l.$
\end{proof}

\subsection{Bounded distortion}
We prove a statement on distortions.
Let $J$ be an interval.
A differentiable map $g\colon J\to \mathbb R$
without a critical point has {\it distortion bounded by} $\kappa\geq1$ if
$$\sup_{x,y\in J}\frac{|Dg(x)|}{|Dg(y)|}\leq\kappa.$$

Let $J\subset T$ be two intervals and $n>0$ such that $f^n|T$ is strictly monotone.
We say $f^nT$ contains a $\xi$-scaled neighborhood of $f^{n}J$ if the lengths 
of both components
of $f^n(T\setminus J)$ are $\geq \xi|f^nJ|$.
The following is known as the Koebe Principle \cite[Chapter IV.1]{dMvS93}.
\begin{lemma}\label{koebe}
Let $J\subset T$ be two intervals and $n>0$ such that $f^n|T$ is strictly monotone and
$f^nT$ contains a $\xi$-scaled neighborhood of $f^nJ$. Then
$f^n|J$ has distortion bounded by $((1+\xi)/\xi)^2$.
\end{lemma}

Let $J\subset \hat X$ be an interval and $n>0$.
We say $f^nJ$ is a {\it free segment} (resp. {\it bound segment}) if it
 is the union of elements of 
$\tilde{\mathcal P}_n$,
and for any $\omega\in\tilde{\mathcal P}_n|J$,
$f^n\omega$ is free (resp. {\it bound}).
A free segment $f^nJ$ is {\it maximal} if it there is no interval $I\subset\hat X$ 
containing $J$ such that $f^nI$ is a free segment.

%If $f^nJ$ is a bound segment, then
%there exists a unique integer $0\leq i<n$ such that $f^iJ$ is a free segment, and 
%$f^jJ$ is a bound segment
%for every $i<j\leq n$.
%We call $i$ the {\it time of creation} of $f^nJ$.

\begin{lemma}\label{bddist}
If $n\geq N$ and $f^nJ$ is a maximal free segment not containing $\{\hat x,-\hat x\}$, then there is an interval $T\supset J$ such that
$f^n|T$ is strictly monotone and $f^nT$ contains a $e^{-6\varepsilon n}$-scaled neighborhood of $f^nJ$.
In particular, $f^n|J$ has distortion bounded by $e^{13\varepsilon n}$.
\end{lemma}

\begin{proof}
By the assumption, to each side of $J$ is attached an interval $\omega\in\tilde{\mathcal P}_n$
such that $f^n\omega$ is bound.
Let $k$ denote the maximal $i<n$ such that $f^i\omega$ is free
and set $p=p_i(\omega)$. Then
$k<n<k+p$.
If $k+p=n+1$, then
using $|Df|\leq 4$ and
 \cite[Lemma 2.6(a)]{ChuTak} we have
   $|f^n\omega|\geq(1/4)|f^{n+1}\omega|=(1/4)|f^{k+p}\omega|
  \geq(1/4)e^{-5\varepsilon p}\geq e^{-6\varepsilon n}$.
  If $k+p>n+1$, then $p=N$ and so
  $|f^n\omega|\geq(1/4)^{k+p-n}|f^{k+p}\omega|\geq(1/4)^N|f^{k+p}\omega|
  \geq(1/4)^Ne^{-5\varepsilon N}\geq e^{-6\varepsilon n}$.
\end{proof}

\subsection{Construction of finite partitions}
The partition $\tilde{\mathcal P}_n$ restricted to $\{R>n\}$
is actually too fine to be used for an upper estimate of the Hausdorff dimension. 
Hence we construct a finite partition $\mathcal P_n$ by gluing some elements of 
$\tilde{\mathcal P}_n$. 

 Start
with $\mathcal P_0=\{\Lambda^-,\Lambda^+\}$. Assume inductively
that $\mathcal P_{n-1}$ has been constructed with the following
properties:

\begin{itemize}

\item[(P1$)_{n-1}$] it is a partition of the set 
$\{R>n-1\}$  into a finite number of intervals each of which
is the union of a countable number of elements of $\tilde{\mathcal P}_{n-1}|\{R>n-1\}$;

\item[(P2$)_{n-1}$] for any $\omega\in\mathcal P_{n-1}$ let
  $$
{\rm fr}(\omega)=\bigcup\{\omega'\in
\tilde{\mathcal P}_{n-1}|\omega\colon f^{n}\omega'\text{ is free}\}\ \ \text{and}\ \
{\rm bo}(\omega)=\bigcup\{\omega'\in
\tilde{\mathcal P}_{n-1}|\omega\colon f^{n}\omega'\text{ is bound}\}.$$
These two sets are intervals unless empty.
In addition, ${\rm fr}(\omega)$ is the union of at most 
$e^{3\varepsilon (n-1)}$ number of elements
of $\tilde{\mathcal P}_{n-1}$.
%Call $\omega=\bar\omega\cup\bar{\bar{\omega}}$, admits a {\it free/bound decomposition} into two mutually disjoint intervals
\end{itemize}

Let $\omega\in \mathcal P_{n-1}$ and write $\{R=n\}=\{R>n-1\}
\setminus\{R>n\}$. The partition $\mathcal P_{n}$ 
on $\omega\setminus\{R=n\}$
is defined as follows.
Let ${\rm bo}(\omega)\in\mathcal P_{n}$ unless empty. If 
${\rm fr}(\omega)\neq\emptyset$ then there
are two cases:

\begin{itemize}
\item if ${\rm fr}(\omega)\cap\{R=n\}=\emptyset$, then define
$\mathcal P_n|{\rm fr}(\omega)$ by dividing $f^{n}{\rm fr}(\omega)$ into at
most two intervals, one which is at the right of $0$ and the
other at the left of $0$;

\item if ${\rm fr}(\omega)\cap\{R=n\}\neq\emptyset$, then
${\rm fr}(\omega)\setminus\{R=n\}$ consists of at most three
intervals ${\omega}^-$, $\omega^+$, $\omega^0$, 
where the corresponding $f^n$-images are: at the left of
$\Lambda^-$; at the right of
$\Lambda^+$; in between
$\Lambda^-$ and $\Lambda^+$. Let $\omega^\pm\in\mathcal
P_{n}$ unless empty.
 Finally define $\mathcal P_n|\omega^0$ by dividing
$f^n\omega^0$ into at most two intervals, one which
is at the right of $0$ and the other at the left of $0$.
\end{itemize}

This completes the definition of $\mathcal P_{n}$. (P1$)_n$ holds
 by construction. 
To see (P2$)_n$, let $\omega\in\mathcal P_n$.
The subdivision algorithm described in Sect.\ref{subdivide}
and the ``monotonicity" of the bound periods
with respect to the distance to the critical point imply that 
${\rm fr}(\omega)$, ${\rm bo}(\omega)$ are intervals or empty sets.
By construction, ${\rm fr}(\omega)\in\tilde{\mathcal P}_{n}$, or else
it is made up of elements of $\tilde{\mathcal P}_{n}$
with the same latest bound period at the same time $k$, $k<n+1$.
Hence
${\rm fr}(\omega)$ is the union of at most 
$e^{3\varepsilon n}$ elements
of $\tilde{\mathcal P}_{n}$.

\subsection{Abundance of long free segments} 
The next lemma allows us to find long free segments in generic partition elements.

\begin{lemma}\label{Marsub}
There exists $k_1=k_1(\delta)>0$ such that
if $k\geq k_1$ and $\omega\in\mathcal
P_k$, then there exist  $q\in[k+1, \left(1+3\varepsilon/\lambda \right)k]$ and 
$\omega'\in\tilde {\mathcal P}_q|\omega$ such that:

\begin{itemize}

\item[(a)] $f^{q}\omega'$ is free;

\item[(b)] $|\omega'|\geq e^{-18\varepsilon k}|\omega|$;

\item[(c)] $\omega'\subset\{R>q-1\}$.

\end{itemize}
\end{lemma}
\begin{proof}
We first consider the case $|{\rm fr}(\omega)|\geq(1/2)|\omega|$.
By (P2$)_k$, ${\rm fr}(\omega)$ is the union of at most $e^{3\varepsilon k}$ number 
of elements of $\tilde{\mathcal P}_k$.
Hence it is possible to choose $\omega'\in\tilde{\mathcal P}_k|{\rm fr}(\omega)$
such that $|\omega'|\geq (1/2)e^{-3\varepsilon k}|\omega|$.
Set $q=k+1$. Then (a) (b) hold. (c) is because $\omega\subset \{R>k\}=\{R> q-1\}$.

%If $\{R=q\}\cap\bar\omega\neq\emptyset$,
%then choose $\omega'\in\mathcal Q|\bar\omega$ 
%such that $R(\omega')=q$.
%By the bounded distortion,
%$$\frac{|\omega'|}{|\omega|}\geq\frac{1}{2}
%\frac{|\omega'|}{|\bar\omega|}\geq e^{-\varepsilon k}\frac{|f^{q}\omega'|}{|f^{q}\bar\omega|}\geq
%e^{-\varepsilon k} \frac{|\Lambda^+|}{|X|}\geq.$$

%If $\{R=q\}\cap\bar\omega=\emptyset$, then
%either $0\notin f^{q}\bar\omega$
% and $\mathcal P_{q}|\bar\omega=\{\bar\omega\}$, or 
% $0\in f^{q}\bar\omega$ and
% $\mathcal P_{q}|\bar\omega=\{\omega_{+1},\omega_{-1}\}$, where
% $f^{q}\omega_{+1}$ (resp.  $f^{q}\omega_{-1}$) is at the right (resp. left) 
% of the critical point $0$.
%Define
%$$\omega'=\begin{cases}
%\bar\omega\ \text{if $0\notin f^{q}\bar\omega$;}\\
%\omega_{\sigma}\ \text{if $0\in f^{q}\bar\omega$, where $\sigma\in\{+1,-1\}$ and $|\omega_{\sigma}|\geq|\omega_{-\sigma}|$}
%\end{cases}$$
%Then $\omega'\in\tilde{\mathcal P}_k|\omega$, $f^k\omega'$ is free
%and $|\omega'|\geq(1/2)|\bar\omega|\geq(1/4)|\omega|\geq e^{-5\varepsilon k}|\omega|$.

 We now consider the case $|{\rm bo}(\omega)|\geq(1/2)|\omega|$.
 In this case we shall choose $\omega'$ to be a certain subinterval of ${\rm bo}(\omega)$.
%The point is that the waiting time needed for some parts
%of  the iterates of 
%$f^{k+1}{\rm bo}(\omega)$ to become free is of order $\leq\varepsilon k$.
% Let $q>k+1$ denote the minimal such that ${\rm bo}(\omega)\notin\mathcal P_{q}$.
Let $i$ denote the maximal $j\leq k$ such that 
$f^{j}{\rm bo}(\omega)$ is a free segment.
Let $r$ denote the minimum of the bound period $p_i\colon 
\tilde{\mathcal P}_i|{\rm bo}(\omega)\to\mathbb N$ at time $i$.  
Set 
$q=i+r$.

\begin{sublemma}\label{qr}
$r\leq k$ and $k+1<q\leq(1+3/\lambda)k$.
\end{sublemma}
\begin{proof}
Since $f^{k+1}{\rm bo}(\omega)$ is a bound segment,
$k+1<q$. To show the rest,
for $i-1\leq j\leq k+1$
let $\omega_j$ denote the element of $\mathcal P_j$ containing ${\rm bo}(\omega)$.
We have $\omega_{i-1}\supset\omega_i\supset\cdots\supset\omega_k\supset\omega_{k+1}$,
$\omega_k=\omega$ and $\omega_{k+1}={\rm bo}(\omega)$.
Note that 
$\tilde{\mathcal P}_j|\omega_j=\{\omega'\in\tilde{\mathcal P}_i|\omega_i\colon p_i(\omega')>j-i\}.$
We treat two cases separately.

\noindent{\it Case I: $\omega_i={\rm bo}(\omega)$.}
Since $f^i{\rm bo}(\omega)$ is a free segment,
$\omega_i\subset{\rm fr}(\omega_{i-1})$.
Since ${\rm fr}(\omega_{i-1})$ is the union of elements of $\tilde{\mathcal P}_{i-1}$,
for any point $x$ in the boundary of ${\rm fr}(\omega_{i-1})$  we have
  $|f^ix|\geq\delta e^{-\varepsilon i}$. In other words, $f^i{\rm fr}(\omega_{i-1})$ is not contained in 
  $(-\delta e^{-\varepsilon i},\delta e^{-\varepsilon i})$, and
the same holds for  $f^i\omega_i$.
By Lemma \ref{P}(b) we have $r\leq
(3\varepsilon/\lambda) i<k$,
 and so
$q\leq k+ r\leq\left(1+3\varepsilon/\lambda
\right)k$ for sufficiently large $k$.

\noindent{\it Case II: $\omega_i\supsetneq{\rm bo}(\omega)$.}
Let $r'$ denote the mimimum of the bound period
$p_i\colon 
\tilde{\mathcal P}_i|\omega_i\to\mathbb N$ at time $i$.
If $i+r'\leq k+1$, then the monotonicity of the bound period implies
$\omega_i=\cdots=\omega_{i+r'-1}\supsetneq\omega_{i+r'}\supsetneq\cdots\supsetneq\omega_{k+1}$,
and $\omega_{k+1}\notin\mathcal P_{k+2}$.
This implies $r=k+2-i$, and so $q=k+2\leq(1+3\varepsilon/\lambda)k$.
If $i+r'>k+1$, then ${\rm fr}(\omega)\cap \omega_i=\emptyset$,
and thus all ${\rm fr}(\omega)$, ${\rm bo}(\omega)$, $\omega_i$ share exactly one boundary point.
Since $f^i$ sends $\omega$ diffeomorphically onto its image,
$f^i{\rm bo}(\omega)$ must come close to the boundary of $(-\delta,\delta)$ so that
$r=N$. Hence 
$q\leq(1+3\varepsilon/\lambda)k$. \end{proof}

Choose $\omega'\in \tilde{\mathcal P}_i|{\rm bo}(\omega)$ 
such that $r=p_i(\omega')$. Then
$f^{q}\omega'$ is free, and $\omega'\subset\{R>q-1\}$.
As for (b),
since $f^i\omega'$ contains
some $ I_{r,j}$ we have  $|f^i\omega'|\geq(\delta_{r-1}-\delta_{r})e^{-3\varepsilon r}$.
Since $|f^i{\rm bo}(\omega)|\leq
\delta_{r-1}$
and
$\delta_{r}\leq
e^{-\frac{\varepsilon}{2}}\delta_{r-1}$,
$$\frac{|f^i\omega'|}{|f^i{\rm bo}(\omega)|}\geq 
\frac{(\delta_{r-1}-\delta_{r})e^{-3\varepsilon r}}
{\delta_{r-1}}\geq(1-e^{-\frac{\varepsilon}{2}})
e^{-3\varepsilon r}.$$ 
Suppose that ${\rm bo}(\omega)$ is contained in an interval
which does not contain $\{\pm\hat x\}$ and whose $f^i$-image
is a maximal free segment.
By Lemma \ref{bddist} and $r\leq k$ in Sublemma \ref{qr},
$$|\omega'| \geq 
e^{-13\varepsilon i}(1-e^{-\frac{\varepsilon}{2}})e^{-3\varepsilon r}|{\rm bo}(\omega)|\geq 
e^{-17\varepsilon k}|{\rm bo}(\omega)|\geq
(1/2)e^{-17\varepsilon k}|\omega|\geq
e^{-18\varepsilon k}|\omega|
.$$
Even if the above is not the case,
the proof of Lemma \ref{bddist} implies essentially the same distortion bounds,
and so the same lower estimate of $|\omega'|$ holds.
\end{proof}

\subsection{Special property of the partition}\label{disti}

 The next lemma asserts that a positive definite fraction of points in each element of
$\mathcal P_k$ quickly return to the base $\Lambda$.
\begin{lemma}\label{Mar}
There exists $k_2\geq \max\{k_0,k_1\}$ such that
if $k\geq k_2$ and $\omega\in\mathcal
P_k$, then there exists $\tilde\omega\in\mathcal Q$ such that:

\begin{itemize}

\item[(a)] $\tilde\omega\subset\omega$ and $|\tilde\omega|\geq e^{-\sqrt{\varepsilon} k}|\omega|$;

\item[(b)] $k< R(\tilde\omega)\leq(1+19\varepsilon/\lambda)k$.
\end{itemize}
\end{lemma}

\begin{proof}
Choose $q\in[k+1, \left(1+3\varepsilon/\lambda\right)k]$  and $\omega'\in\tilde{\mathcal P}_q|\omega$ for which the conclusions of
Lemma \ref{Marsub} holds. 
By Lemma \ref{escape}, 
 $$|\omega'\cap\{R<q+(16\varepsilon/\lambda)
k\}|\geq(1-C_1\zeta^{k})|\omega'|\geq (1/2)|\omega'|.$$
  Since the $f^r$-image of $f^q\omega'$ is folded at most $2^r$ times,
    $\#\{\omega\in\mathcal Q|\omega'\colon R(\omega)=q+r\}\leq 2^{r+1}$
  and so
 \begin{align*}
 \#\{\omega\in\mathcal Q|\omega'\colon R(\omega)<q+ (16\varepsilon/\lambda) k\}
 &=\sum_{r=0}^{[(16\varepsilon/\lambda) k]}\#\{\omega\in\mathcal Q|\omega'\colon R(\omega)=
 q+r\}\\&\leq
 \sum_{r=0}^{[(16\varepsilon/\lambda) k]}2^{r+1}\leq e^{(17\varepsilon/\lambda) k},
 \end{align*}
 where the last inequality holds for sufficiently large $k$.
 Then it is possible to choose $r\leq (16\varepsilon/\lambda) k$ and $\tilde\omega\in\mathcal Q|\omega'$ such
that $R(\tilde\omega)=q+r \leq(1+19\varepsilon/\lambda)k$ and
$|\tilde\omega|\geq(1/2)
e^{-(17\varepsilon/\lambda) k}|\omega'|.$ From this and
Lemma \ref{Marsub}(b) we obtain (a).
\end{proof}

\subsection{Towers}
We now translate Lemma \ref{Mar} into the language of towers.
By Lemma \ref{escape}, we may think of $R\colon\mathcal Q\to \mathbb N$ as
a function on a full measure subset of $\Lambda$ in the obvious
way. 
Let
$$\Delta=\{(x,\ell)\colon x\in \Lambda,\ \ \ell=0,1,\ldots,R(x)-1\},$$
which we call a {\it tower}, and define a {\it tower map} $\hat f\colon
\Delta\circlearrowleft$ by
$$\hat f(x,\ell)=\begin{cases}(x,\ell+1)\ \ \text{ if }\ell+1<R(x); &\\
(f^{R(x)} x,0)\ \text{ if }\ell+1=R(x).\end{cases}$$ The point
$(x,\ell)$ is considered to be climbing the tower in the first case, and
falling down from the tower in the second case.
 Define
 $\Delta_\ell=\{(x,\ell)\in\Delta\colon R(x)>\ell\}.$ 
With the canonical identification $\{R>\ell\}\ni x\mapsto(x,\ell)\in\Delta_\ell$ 
we transplant the partition $\mathcal P_\ell$ of $\{R>\ell\}$ to the partition of $\Delta_\ell$ and also
denote it by $\mathcal P_\ell$. Let $\mathcal D=\bigcup_{\ell\geq0}\mathcal
P_\ell$. This  is a partition of $\Delta$ with a Markov property: for any 
$\omega\in\mathcal D$, $\hat f\omega$ is a finite union of elements of $\mathcal D$.
We identify $\Delta_0=\{(x,0)\colon x\in \Lambda\}$ with $\Lambda$ under the action
of the map $\pi\colon \Delta\to\Lambda$ given by $\pi(x,\ell)=x$.

\begin{lemma}\label{towerreturn}
The following holds for sufficiently large $n$: for 
any $A\in\bigvee_{i=0}^{n-1}\hat f^{-i}\mathcal D$ 
with $A\subset \Delta_0$ 
there exist an interval
$\tilde A\subset\Delta_0$ and $t\in
[\left(1-\varepsilon\right)n,\left(1+20\varepsilon/\lambda\right)n]$ such that:

\begin{itemize}

\item[(a)] $\tilde A\subset A$, and $\hat f^{t}\tilde A
=\Lambda^+$ or $=\Lambda^-$;

\item[(b)] $|\tilde A|\geq e^{-2\sqrt{\varepsilon} n}|A|$.

\end{itemize}
\end{lemma}
\begin{proof}
In the first $n$-iterates under $\hat f$,
the interval $A$ continues climbing the tower, or else falls down from the tower several times.
Let $j=\max\{i\geq0\colon \hat f^iA\subset\Delta_0\}$.
Since $f^jA\subset\Lambda^\pm$ and $f^j|A$ is extended to a diffeomorphism
onto $3\Lambda^\pm$, $f^j|A$ has distortion bounded by $4$.
Set $\omega=f^jA$ and $k=n-j-1$.
Since $A\subset\Delta_0$  we have $\omega\in\mathcal P_{k}$.
If $k\geq k_2$, then
take
a subinterval $\tilde\omega\subset \omega$ for which the conclusions of Lemma
\ref{Mar} hold, and define $t=j+R(\tilde\omega).$ The bounds on $t$ 
follow from Lemma \ref{Mar}(b).
Define $\tilde A$ to be the subinterval of $A$ be such that
$\tilde\omega=f^j\tilde A$. Then
\begin{equation}\label{p11}
\frac{|\tilde A|}{|A|}\geq
\frac{1}{4}\frac{|\tilde\omega|}{|\omega|}\geq
\frac{1}{4}e^{-\sqrt{\varepsilon} k} \geq
e^{-2\sqrt{\varepsilon} n}.\end{equation} 
The second inequality follows from Lemma \ref{Mar}(a).
If $k<k_2$, then set
$t=j$ and
 define $\tilde A$ to be the subinterval of $A$ such that
 $f^t\tilde A=\Lambda^\pm$. 
\end{proof}

\section{Upper estimate of Birkhoff spectrum}\label{spectrum}
We put together the constructions and the results in Sect.2 
to obtain an upper estimate of the Birkhoff spectrum. 
For $\varphi\in C(X)$, $k\geq0$, $\alpha\in[c_\varphi,d_\varphi]$
and $\varepsilon>0$
consider the set
$$\Gamma_{k}
=\Gamma_{k}(\varphi;\alpha , \varepsilon)
=\left\{x\in \Lambda\colon \left|\frac{1}{n}S_n\varphi(x)-
\alpha\right|<\varepsilon\ \ \text{for every } n\geq k\right\}.$$
Note that $\Gamma_k$ is increasing in $k$.
Since $K_\varphi(\alpha)$ is dense in $X$ and $\Lambda$ contains open sets,
$\Gamma_k\neq\emptyset$ holds for sufficiently large $k$.
Define $\sigma=\sigma(\varphi; \alpha ,\varepsilon)$ by
\begin{equation}\label{defsigma}
\sigma={\sup}\left\{\frac{h(\mu)}{\lambda(\mu)}\colon\mu\in\mathcal M_f,\
|\mu(\varphi)-\alpha|\leq \sqrt\varepsilon\right\}+\varepsilon^{\frac{1}{3}}.
\end{equation}
Since $\lambda_{\rm inf}>0$,  $\sigma$ stays bounded from above as $\varepsilon\to 0$.

\begin{prop}\label{upperd}
If $\varphi\in C(X)$ is Lipschitz, then for any $\alpha\in[c_\varphi,d_\varphi]$
and $\varepsilon>0$,
$$\dim_H\Gamma_{k}(\varphi;\alpha,\varepsilon)\leq\sigma(\varphi;\alpha,\varepsilon)
\quad\text{for every $k\ge 0$.}$$
%Let $\varphi\colon X\to\mathbb R$ be Lipschitz continuous 
%and $\alpha\in[c_\varphi,d_\varphi]$.
%For any $\varepsilon>0$ there
%exists $\xi\in\mathcal M_f$ such that 
%\begin{equation*}\label{upperd1}
%\dim_H\Gamma_{k}(\varphi;\alpha,\varepsilon)\leq\sigma(\varepsilon).
%\end{equation*}
\end{prop}

We finish the upper estimate of $B_\varphi(\alpha)$ assuming the conclusion of Proposition \ref{upperd}. Set $Y=[f^20,f0]$. Points in $X\setminus Y$ are mapped to $Y$ 
by some positive iterates. 
The countable stability and the invariance of
Hausdorff dimension under  the action of
Lipschitz continuous homeomorphisms yields
$B_\varphi (\alpha)
= \dim_H  (K_\varphi(\alpha)\cap Y).$
We estimate the right-hand-side.

By (A4) there exists $M>0$ such that
$f^{M}\Lambda=Y$. Then
$$\dim_H (K_\varphi(\alpha)\cap Y)
= \dim_Hf^M(K_\varphi(\alpha)\cap\Lambda )
\leq \dim_H (K_\varphi(\alpha)\cap\Lambda )
\leq\lim_{k\to\infty}\dim_H \Gamma_{k},$$
where the last inequality is because
%$K_\varphi\cap
%\Lambda\subset \bigcap_{n\geq0}\bigcup_{k\geq n}\Gamma_{k}(\alpha).$
$K_\varphi(\alpha)\cap
\Lambda\subset\bigcup_{k\geq n}\Gamma_{k}$ for every $n\geq0$.
If $\varphi$ is Lipschitz continuous, then by Proposition \ref{upperd},
\begin{align*}
B_\varphi(\alpha) \leq \lim_{k\to\infty}\dim_H \Gamma_{k} \leq \sigma.
\end{align*}
 Letting
$\varepsilon\to0$ we get
$$B_\varphi (\alpha) 
\leq \lim_{\varepsilon\to0}{\sup}
\left\{\frac{h(\mu)}{\lambda(\mu)}\colon\mu\in
\mathcal M_f,\ |\mu(\varphi)-\alpha|<\varepsilon\right\} .$$ 

If $\varphi$ is merely continuous, then take
a Lipschitz continuous $\tilde\varphi$ such that
$\|\varphi - \tilde\varphi\| <\varepsilon / 2 ,$ 
$c_{\tilde\varphi} = c_{\varphi}$ and $d_{\tilde\varphi}=d_{\varphi} .$
Then for any
 $\alpha\in[c_{\varphi},d_{\varphi}]$
and small $\varepsilon >0$,
$\Gamma_{k}(\varphi;\alpha , \varepsilon) \subset 
\Gamma_{k}(\tilde\varphi; \alpha , 2\varepsilon)$
holds.
By Proposition \ref{upperd} there exists $\xi\in\mathcal M_f$ 
such that
$$\dim_H \Gamma_{k}(\varphi;\alpha , \varepsilon) 
\le \dim_H \Gamma_{k}( \tilde\varphi; \alpha , 2\varepsilon) \leq 
\frac{h(\xi)}{\lambda(\xi)} + (3\varepsilon)^{\frac{1}{3}},$$
and
$$|\xi (\varphi) - \alpha | \le |\xi(\varphi) - \xi(\tilde\varphi)| +   |\xi
(\tilde\varphi) - \alpha| 
< \varepsilon /2 + \sqrt{2\varepsilon}
< 2\sqrt\varepsilon.$$
The rest of the argument is identical to the previous case.

\medskip
The rest of this section is entirely devoted to the proof of Proposition
\ref{upperd}. In Sect.\ref{conh}
we extract from the towers uniformly hyperbolic invariant sets (horseshoes).
In Sect.\ref{construct2} we construct invariant measures on the horseshoes,
and use them to complete the proof of the proposition.

\subsection{Construction of a horseshoe}\label{conh}
%Let $\varphi\in C(X)$ be Lipschitz continuous, $\alpha\in[c_\varphi,d_\varphi]$
%and $\varepsilon>0$.
 Define
$$\mathcal A_n=\left\{A\in\bigvee_{i=0}^{n-1}\hat f^{-i}\mathcal D\colon
A\subset\Delta_0,\ \
\left|\frac{1}{n}S_n\varphi(x)-\alpha\right|<\varepsilon\
\ \text{for some }x\in A\right\}.$$
If $\Gamma_k \neq\emptyset$ then
for every $n\geq k$ we have $\mathcal A_n\neq\emptyset$, and
$\Gamma_{k}\subset\bigcup_{A\in\mathcal
A_n} A.$ We use this family of coverings for the upper estimate of the Hausdorff dimension.

Let $\Omega$ be a finite collection of pairwise disjoint
closed intervals in $\Lambda$ and $r$ a positive integer.
We say $\Omega$ \emph{generates a horseshoe for} $f^r$ if
$f^r$ sends each element of $\Omega$ diffeomorphically onto 
$\hat X$.
By a {\it horseshoe} we mean the set
$$H_r(\Omega)=\bigcap_{j=0}^\infty(f^{r})^{-j}
\left(\bigcup_{I\in\Omega} I\right).$$

\begin{lemma}\label{horse2}
For any $\varepsilon>0$ there exists $n'>0$ such that if $n\geq n'$ then $\mathcal A_n\neq\emptyset$ and
there exist a finite
 collection $\mathcal K$ of closed intervals in $\Lambda$
 and an integer $q\in[(1-\varepsilon)n,(1+21\varepsilon/\lambda) n]$
 such that:

\begin{itemize}
\item[(a)] $\mathcal K$ generates a horseshoe for $f^{q}$;

\item[(b)] $\sum_{K\in\mathcal K} |K|^{\sigma}\geq
e^{-3\sqrt{\varepsilon}\sigma n}\sum_{A\in\mathcal A_n}
|A|^{\sigma}$;

\item[(c)] for all $x\in H_q(\mathcal K)$, $|(1/q)
S_{q}\varphi(x)-\alpha|\leq\sqrt{\varepsilon}$.
\end{itemize}
%where $\sigma=\sigma(\varphi; \alpha ,\varepsilon)$ 
%is as in (\ref{defsigma}).
\end{lemma}
\begin{proof}
For each $A\in\mathcal A_n$, fix once and for all
 an interval $\tilde A$ and an integer $t=t_A$ for which the 
conclusions of Lemma \ref{towerreturn} hold. 
Let $\mathcal A_n(t)=\{A\in\mathcal A_n\colon t_A=t\}$.
Then
$t_A\in[(1-\varepsilon)n, (1+19\varepsilon/\lambda)n]$.
Let $t_0$ be a value of $t$ which maximizes
$\sum_{A\in\mathcal A_n(t)}|\tilde A|^{\sigma}$.
Then
\begin{equation}\label{horse3}
\sum_{A\in\mathcal
A_n(t_0)}|\tilde A|^{\sigma}\geq\frac{1}{(1+20/\lambda)\varepsilon
n}\sum_{A\in\mathcal A_n}|\tilde A|^{\sigma}.\end{equation} 
By (A4)
it is possible to choose a constant $\tau>0$, an integer $u>0$ and a closed interval
$I^+\subset\Lambda^+$ such that $\Lambda^+$ contains the $\tau$-scaled neighborhood
of $I^+$, and $f^{u}$ sends $I^+$
diffeomorphically onto $\hat X$. 
 Define $q=t_0+u$. The bounds on $q$ hold
 for sufficiently large $n$.

Let $I^-=-I^+$. 
For each $A\in\mathcal A_n(t_0)$ define $K(A)$ to be the preimage
of $I^+$ or $I^-$ under $f^{t_0}|\tilde{A}$, according to whether
$f^{t_0}\tilde {A}=\Lambda^+$ or $=\Lambda^-$. Set
$\mathcal K=\{K(A)\colon A\in\mathcal A_n(t_0)\}$. Then $\mathcal K$
is a finite collection of pairwise disjoint
closed intervals in $\Lambda$, and $f^{q}$ sends each element of $\mathcal K$ diffeomorphically 
onto $\hat X$. 
Set $c=(\tau/(1+\tau))^2|I^+|/|\Lambda^+|$.
Then 
\begin{align*}\sum_{K\in\mathcal K}|K|^{\sigma}&\geq
c^{\sigma}\sum_{A\in\mathcal A_n(t_0)}
|\tilde{A}|^{\sigma}
\geq
\frac{c^{\sigma}}{(1+20/\lambda)
\varepsilon n}\sum_{A\in\mathcal
A_n}|\tilde A|^{\sigma}\\&\geq\frac{1}{(1+20/\lambda)\varepsilon n}(ce^{-2\sqrt{\varepsilon}
n})^{\sigma}\sum_{A\in\mathcal A_n}|A|^{\sigma}
\geq e^{-3\sqrt{\varepsilon}\sigma n}\sum_{A\in\mathcal A_n}|A|^{\sigma}
.\end{align*}

\begin{sublemma}\label{bdd}
There exists a constant $C>0$ such that
if $n\geq N$ and $\omega\in\mathcal P_{n-1}$, then for all $x,y\in\omega$,
 $$|S_n\varphi(x)-S_n\varphi(y)|\leq {\rm Lip}(\varphi)\cdot C\delta^{-1},$$
 where ${\rm Lip}(\varphi)$ denotes the Lipschitz constant of $\varphi$.
\end{sublemma}
\begin{proof}
Let $0\leq i\leq n-1$. We call $f^i\omega$ free if there exists an interval 
$J\subset\hat X$
containing $\omega$ such that $f^iJ$ is a free segment. 
%Otherwise we call $f^i\omega$ bound.
Let $i_0$ denote the maximal $i\leq n-1$ such that $f^i\omega$ is free.
From the construction in Sect.2 one can find integers $0\leq r_1<\cdots<r_s=i_0$,
$p_1,\ldots,p_s$
such that:
$r_{1}$ is the smallest $i\geq 0$ with $f^i\omega\cap(-\delta,\delta)\neq\emptyset$;
$\delta_{p_k}\leq d(0,f^{r_k}\omega)\leq\delta_{p_k-2}$ and
$r_{k+1}$ is the smallest $i\geq r_k+p_k$ with 
$f^i\omega\cap(-\delta,\delta)\neq\emptyset$
$(k=1,\ldots,s-1)$;
$d(0,f^{r_s}\omega)\leq\delta_{p_s-2}$ and
$n\leq r_s+p_s$.
Then,
similarly to the proof of \cite[Sublemma 3.16]{ChuTak}
one can show that 
$$\sum_{i=0}^{n-1}|f^i\omega|
\leq C\delta^{-1}.$$
This implies the desired inequality since $\varphi$ is Lipschitz continuous.
\end{proof}

To prove (c), for each $A\in\mathcal A_n(t_0)$ pick $x_A\in A$
such that
$|(1/n)S_n\varphi(x_A)-\alpha|\leq \varepsilon$. We have 
$$S_q\varphi(x_A)\geq S_n\varphi(x_A)-
\sup|\varphi|\cdot|q-n|\geq \alpha n-\sup|\varphi|\cdot(21\varepsilon/\lambda) n
\geq
\left(\alpha-\sqrt{\varepsilon}/2\right)q.$$
In the same way we have
$S_q\varphi(x_A)\leq\left(\alpha+\sqrt{\varepsilon}/2\right)q.$
 Then
\begin{equation}\label{on1}\left|S_q\varphi(x_A)-q\alpha\right|\leq
\sqrt{\varepsilon}\cdot q/2.\end{equation} By
Sublemma \ref{bdd}, for any $x\in A$ we have
$|S_{t_0}\varphi(x_A)-S_{t_0}\varphi(x)|\leq {\rm Lip}(\varphi)\cdot
C\delta^{-1}$. 
Since $q-t_0=u$ and 
 $q\geq(1-\varepsilon)n$,
for sufficiently large $n$
we have
\begin{equation}\label{on2}
|S_q\varphi(x_A)-S_q\varphi(x)|\leq {\rm Lip}(\varphi)
\cdot C\delta^{-1}+2\sup|\varphi|\cdot (q-t_0)\leq\sqrt{\varepsilon}\cdot q/2
.\end{equation} 
(\ref{on1}) (\ref{on2}) yield
$|S_{q}\varphi(x)-q\alpha|\leq \sqrt{\varepsilon}\cdot q$.
\end{proof}

\subsection{Construction of a measure on the horseshoe}\label{construct2}
%We construct a measure $\xi$ supported on the horseshoe.
For sufficiently large $n$, 
choose a finite collection $\mathcal K$ of closed intervals in $\Lambda$ 
and a positive integer $q$ for which the conclusions of 
Lemma \ref{horse2} hold. 
Set $F=f^q$.
By construction, $F$ is uniformly expanding on each element of $\mathcal K$.
 Hence, $F|H_q(\mathcal K)$ is topologically conjugate to
the one-sided full shift on $\#\mathcal K$-symbols.
%Define 
%a continuous function $\Phi\colon H_{q}(\mathcal K)\to\mathbb R$ by
%$\Phi(x)=\sigma\log |DF(x)|$.
% Let $\nu_{\Phi}$ denote an equilibrium state of $F$
%for the potential $-\Phi$, namely an $F$-invariant measure such that
%$$h_F(\nu_{\Phi})-\nu_{\Phi}(\Phi)=\sup\left\{h_F(\nu)-\nu(\Phi)
%\colon\text{$\nu$ is $F$-invariant}\right\}.$$ Here, $h_F(\nu)$
%denotes the entropy of $(F,\nu)$.
%Let $\xi=(1/q)\sum_{i=0}^{q-1}(f^i)_*\nu_{\Phi}$, which is
%$f$-invariant. Note that
%$\xi(\varphi) =(1/q)\nu_{\Phi}(S_{q}\varphi)$.
%From Lemma \ref{horse2}(c) it follows
%that $\left|(1/q)S_{q}\varphi-\alpha\right|\leq\sqrt\varepsilon$ 
%$\nu_{\Phi}$-a.e., and thus 
%$\left|\xi(\varphi)-\alpha\right|\leq\sqrt\varepsilon$.
Write $\mathcal K=\{K_1,\ldots,K_{\#\mathcal K}\}$.
For $\ell>0$ and an $(\ell+1)$-string $(a_0,\ldots,a_\ell)$
of integers in $[1,\#\mathcal K]$, define an interval $$K_{a_0\cdots
a_{\ell}}=K_{a_0}\cap F^{-1}K_{a_1}\cap\cdots \cap F^{-\ell}K_{a_\ell}.$$
Set $\kappa=2C_0\sup_{x,y\in I^+}\frac{|Df^u(x)|}{|Df^u(y)|}$. We have
%By the Koebe Principle,
%there exists $\kappa\geq 1$ such that
%for any intervals $J\subset T$ such 
%that $f^n|T$ is strictly monotone, $f^nT=3\Lambda^\pm$ and $f^nJ\subset\Lambda^\pm$,
%$f^n|J$ has distortion bounded by $4$.
%Hence
$$\frac{|K_{a_0\cdots a_{\ell}}|}{|K_{a_0\cdots a_{\ell-1}}|}
\geq\frac{1}{2}\frac{|F^{\ell-1}K_{a_0\cdots a_{\ell}}|}{|F^{\ell-1}K_{a_0\cdots a_{\ell-1}}|}
=\frac{1}{2}\frac{|\{x\in K_{a_{\ell-1}}\colon Fx\in K_{a_{\ell}}\}|}{|K_{a_{\ell-1}}|}
\geq\kappa^{-1}|K_{a_\ell}|.$$
The first inequality follows from the Koebe Principle, and
the second one from \eqref{bounddist} and the definition of $I^+$, $u$.
Then
\begin{align*}
\sum_{(a_0,\ldots, a_\ell)} |K_{a_0\cdots a_\ell}|^{\sigma}&=
\sum_{(a_0,\ldots, a_{\ell-1})}|K_{a_0\cdots
a_{\ell-1}}|^{\sigma}\sum_{a_\ell}\frac{
|K_{a_0\cdots a_{\ell}}|^{\sigma}}
{|K_{a_0\cdots a_{\ell-1}}|^{\sigma}}\\
&\geq \kappa^{-\sigma}\sum_{i=1}^{\#\mathcal K}|K_i|^{\sigma}\sum_{(a_0,\cdots,
a_{\ell-1})}|K_{a_0\cdots a_{\ell-1}}|^{\sigma}\\
&\geq\cdots\geq
\left(\kappa^{-\sigma}\sum_{i=1}^{\#\mathcal K}|K_i|^{\sigma}\right)^{\ell+1}.\end{align*}
This yields
\begin{equation}\label{lem1}\varliminf_{\ell\to\infty}\frac{1}{\ell}
\log\sum_{(a_0,\ldots, a_{\ell})}|K_{a_0\cdots
a_{\ell}}|^{\sigma}\geq\log\sum_{i=1}^{\#\mathcal K}
|K_i|^{\sigma}-\sigma\log
\kappa.\end{equation}
Let $\nu_{a_0\cdots a_\ell}$ denote the uniform distribution
on the orbit of the $(\ell+1)$-periodic point of $F$ in $K_{a_0\cdots a_\ell}$.
Define an $F$-invariant probability measure $\nu_\ell$ 
supported on $H_q(\mathcal K)$ by
$$\nu_\ell=\rho_\ell\sum_{(a_0,\ldots, a_\ell)}
|K_{a_0\cdots a_\ell}|^{\sigma}
\nu_{a_0\cdots a_\ell},$$
where $\rho_\ell=1/\sum_{(a_0,\ldots, a_\ell)}
|K_{a_0\cdots a_\ell}|^{\sigma}$ is the normalizing constant.
%\ \ \text{where}\ \ \rho_k=\left(\sum_{(a_0,\ldots,a_k)}
%{\rm length}[a_0\cdots a_k]^{\sigma}\right)^{-1}.$$ 
Pick an accumulation point of the
sequence $\{\nu_\ell\}$ and denote it by $\nu_0$. Taking
a subsequence if necessary we may assume this convergence
takes place for the entire sequence. Using the relation
 $\nu_\ell(K_{a_0\cdots a_\ell})=\rho_\ell |K_{a_0\cdots a_\ell}|^{\sigma}$
and 
$|K_{a_0\cdots a_\ell}|\leq
\kappa|\hat X| \exp\{-(\ell+1) \int \log |DF| d\nu_{a_0\cdots a_\ell} \}$
%$|K_{a_0\cdots a_\ell}|^{\sigma}\leq
%\kappa^\sigma e^{-(\ell+1)\nu_{a_0\cdots a_\ell}(\Phi)}$ 
which follows from the bounded distortion
 we have
\begin{align*}
\log\sum_{(a_0,\ldots, a_\ell)}|K_{a_0\cdots a_\ell}|^{\sigma}
&=\sum_{(a_0,\ldots, a_\ell)}\nu_\ell(K_{a_0\cdots
a_\ell})\left(-\log\nu_\ell(K_{a_0\cdots a_\ell})+
\sigma\log |K_{a_0\cdots a_\ell}|\right)\\
&\leq -\sum_{(a_0,\ldots, a_\ell)}
\nu_\ell(K_{a_0\cdots
a_\ell})\log\nu_\ell(K_{a_0\cdots a_\ell})
-\sigma (\ell+1) \int \log |DF| d\nu_\ell 
+ \sigma \log \kappa|\hat X|.
\end{align*}
 A slight modification of the argument in the proof of the
Variational Principle \cite[Theorem 9.10]{Wal82} shows that 
%for any integer $p$ with $1\leq p<\ell$,
%\begin{align*}
%\frac{1}{\ell}\log\sum_{(a_0,\ldots, a_\ell)}
%|K_{a_0\cdots a_\ell}|^{\sigma}
%\leq&-\frac{1}{p}\sum_{(a_0,\ldots,a_p)}
%\nu_\ell(K_{a_0\cdots a_p})
%\log\nu_\ell(K_{a_0\cdots a_p})-\frac{\ell+1}{\ell}      
%\sigma  \int \log |F'| d\nu_\ell 
%+\frac{2p\log\#\mathcal K}{\ell}.\end{align*} 
%Letting $\ell\to\infty$ and
%then $p\to\infty$ we get
\begin{equation}\label{lem2}
\varlimsup_{\ell\to\infty}\frac{1}{\ell}\log\sum_{(a_0,\ldots,
a_\ell)}|K_{a_0\cdots a_\ell}|^{\sigma}\leq
h_F(\nu_0)-\sigma\int\log |DF | d\nu_0 ,
\end{equation}
where
$h_F(\nu_0)$ denotes the entropy of $(F,\nu_0 )$.
%(\ref{lem1}) (\ref{lem2}) yield
%\begin{equation*}
%\log\sum_{i=1}^{\#\mathcal K}|K_i|^{\sigma}-\sigma\log \kappa \leq
%h_F(\nu_0)- \sigma \int \log |DF| d\nu_0.\end{equation*}
Let $\xi=(1/q)\sum_{i=0}^{q-1}(f^i)_*\nu_0$, which is
$f$-invariant. 
It follows from Lemma \ref{horse2}(c) 
that
$|\xi(\varphi) - \alpha| \leq \sqrt{\varepsilon}$
since $\nu_0$ is supported on $H_q(\mathcal K)$.
Then by the definition of $\sigma$ in (\ref{defsigma}), 
(\ref{lem1}) (\ref{lem2}) yield
%Using this and $h(\xi)-(\sigma-\sqrt{\varepsilon})\lambda(\xi)\leq0$,
\begin{align*}\log\sum_{i=1}^{\#\mathcal K} |K_i|^{\sigma}
&\leq h_F(\nu_0)-\sigma\int \log|DF|d\nu_0+\sigma\log \kappa \\
&=q(h(\xi)-\sigma\lambda(\xi))+\sigma\log \kappa 
\leq -q\varepsilon^{\frac{1}{3}}\lambda(\xi)+\sigma\log \kappa .\end{align*}
Lemma \ref{horse2}(b) gives
$$\log\sum_{A\in\mathcal A_n}|A|^{\sigma}
\leq
4\sqrt{\varepsilon}\sigma n+\log\sum_{i=1}^{\#\mathcal K}
|K_i|^{\sigma}.$$ Since
$q\geq(1-\varepsilon)n$ and
 $\lambda(\xi)\geq\lambda_{\rm inf}>0$ we have
\begin{align*}
\log\sum_{A\in\mathcal A_n}|A|^{\sigma}
&\leq  4\sqrt{\varepsilon}\sigma n
-q\varepsilon^{\frac{1}{3}}\lambda(\xi)+\sigma\log \kappa \\
&\leq \left(4\sqrt{\varepsilon}\sigma-(1-\varepsilon)\varepsilon^{\frac{1}{3}}\lambda_{\rm inf}\right)n
+\sigma\log \kappa \leq-\varepsilon^{\frac{1}{3}}\lambda_{\rm inf}n/2,\end{align*} 
where the last inequality holds for sufficiently large $n$.
It follows
that $\sum_{A\in\mathcal A_n}|A|^{\sigma}$
has a negative growth rate as $n$ increases. In addition, the above inequality implies
that the diameters of the elements of $\mathcal A_n$ decrease uniformly as $n$ increases.
Therefore
the Hausdorff $\sigma$-measure of $\Gamma_{k}$
is zero and so $\dim_H\Gamma_{k}\leq\sigma$.
This completes the proof of Proposition \ref{upperd}. \qed

%\begin{lemma}\label{size}
%For every $A\in\mathcal
%A_n$ we have ${\rm length}(A)\leq e^{-\frac{\lambda}{4}n}$.
%\end{lemma}
%\begin{proof}
%Since $A\subset\Delta_0$, $\hat f^nA\in\Delta_l$ holds for some
%$0\leq l\leq n$. If $l\geq \frac{N}{\varepsilon}$, then
%let $B_A$ denote the subinterval of $A$ as in Lemma
%\ref{towerreturn}. Using (\ref{p11}) and ${\rm length}(B_A)\leq
%e^{-\frac{\lambda}{3}n}$ we get
%${\rm length}(A)\leq e^{-\frac{\lambda}{4}n}.$ If $l<\frac{N
%}{\varepsilon}$, then for all $x\in A$ we have
%$|(f^{n-l})'x|\geq \delta e^{\frac{\lambda}{3}(n-l)}$, and thus ${\rm length}(A)\leq
%\delta^{-1}e^{-\frac{\lambda}{3}(n-l)}\leq e^{-\frac{\lambda}{4}n}.$
%\end{proof}

\section{Lower estimate and continuity of Birkhoff spectrum}
In this section we estimate $B_\varphi(\alpha)$ from below,
and finish the proof of the formula in Theorem A.  We then use
this formula to prove the continuity of the Birkhoff spectrum.

\subsection{Lower estimate of the Birkhoff spectrum}\label{lowest}
To estimate $B_\varphi(\alpha)$ from below we will construct a sufficiently large
set of points for which the time averages of $\varphi$ are precisely equal to $\alpha$.
Let $\mathcal M_f^e$ denote the set of ergodic elements of
$\mathcal M_f$. 
\begin{prop}\label{lowd}
Let $\varphi\in C(X)$ and $\alpha\in[c_\varphi,d_\varphi]$. 
Let $\{\mu_i\}_i$ be a sequence in $\mathcal M_f^e$ such that
$|\mu_i(\varphi)-\alpha|<1/i$
and  $h(\mu_i)/\lambda(\mu_i)$ converges as $i\to\infty$.
There exists a closed set
$\Gamma\subset K_\varphi(\alpha)$ such that
\begin{equation*}
\dim_H(\Gamma)\geq \lim_{i\to\infty}\frac{h(\mu_i)}{\lambda(\mu_i)}.\end{equation*}
\end{prop}
It then follows that
\begin{equation}\label{HD'}
B_\varphi(\alpha)\geq\lim_{\varepsilon\to0}{\sup}\left\{\frac{h(\mu)}{\lambda
(\mu)}\colon\mu\in\mathcal
M_f^e,\
\left|\mu(\varphi)-\alpha\right|<\varepsilon\right\}.\end{equation}
To finish, it is left to 
show that the supremum of the right-hand-side of \eqref{HD'} 
may be taken over all invariant probability measures
which are not necessarily ergodic.

%Given  $\mu\in\mathcal M_f$, considering the ergodic decomposition
%one can find a linear combination of ergodic measures which approximates
%the entropy, the Lyapunov exponent of $\mu$ and integrals of continuous functions
%against $\mu$. In view of this and \cite{Kat80},
Using (A4) and a one-dimensional version of Katok's theorem \cite{Kat80}, 
for any $\mu\in\mathcal M_f$ and $\varepsilon>0$
one can find $\nu\in\mathcal M_f^e$ such that:
$|\mu(\varphi)-\nu(\varphi)|\leq\varepsilon$;
$h(\nu)\geq h(\mu)-\varepsilon$;
$\lambda(\nu)\leq\lambda(\mu)+\varepsilon$.
Since $0<\lambda_{\rm inf}\leq\lambda(\mu)\leq\log4$ and $h(\nu)\leq\log2$
we have
$$\frac{h(\nu)}{\lambda(\nu)}
\geq \frac{h(\mu) - \varepsilon} {\lambda(\mu) + \varepsilon}
=\frac{h(\mu)}{\lambda(\mu)} 
- \frac{\varepsilon (h(\mu) + \lambda(\mu))} 
{\lambda(\mu)(\lambda(\mu)+\varepsilon)}
\geq\frac{h(\mu)}{\lambda(\mu)}
-\frac{3\varepsilon\log2}{\lambda_{\rm inf}^2}, $$
%$$\frac{h(\nu)}{\lambda(\nu)}
%=\frac{h(\mu)}{\lambda(\mu)}+\frac{\varepsilon(\frac{h(\mu)}{\lambda(\mu)}-1)}{\lambda(\mu)-\varepsilon}
%\geq\frac{h(\mu)}{\lambda(\mu)}-\frac{\varepsilon(\lambda(\mu)+h(\nu))}{\lambda^2(\mu)}\geq\frac{h(\mu)}{\lambda(\mu)}-\frac{3\varepsilon\log2}{\lambda_{\rm inf}^2},$$
and therefore
$${\sup}\left\{\frac{h(\nu)}{\lambda
(\nu)}\colon\nu\in\mathcal M_f^e,\ \left|\nu(\varphi) 
-\alpha\right|<2\varepsilon\right\}\geq
{\sup}\left\{\frac{h(\mu)}{\lambda (\mu)}\colon\mu\in\mathcal
M_f,\ \left|\mu(\varphi)
-\alpha\right|<\varepsilon\right\}
-\frac{3\varepsilon\log2}{\lambda_{\rm inf}^2}.$$ 
Letting $\varepsilon\to0$ and then using
\eqref{HD'} we obtain
\begin{equation*}
B_\varphi(\alpha)\geq\lim_{\varepsilon\to0}{\sup}\left\{\frac{h(\mu)}{\lambda
(\mu)}\colon\mu\in\mathcal M_f,\ \left|\mu(\varphi)
-\alpha\right|<\varepsilon\right\}.\end{equation*}
From this and the upper estimate in Sect.3 we obtain the formula
in Theorem A. \medskip

\noindent{\it Proof of Proposition \ref{lowd}.}
%We construct a Moral fractal $\Gamma\subset K_\varphi(\alpha)$
%such that $\dim_H(\Gamma)\geq D_0$.
%Our strategy is to put a probability measure $\nu_0$ on it and compare $\nu_0(I)$ and ${\rm %length}(I)^{D_0-\varepsilon}$
%for any small interval $I$ and any small $\varepsilon>0$.
If $ h(\mu_i) \to0$ then there is nothing to prove 
since $\lambda (\mu_i) \geq \lambda_{\inf}>0$.
So we may assume $h(\mu_i)>0$ for each $i$.
By a result of \cite{Kat80}, for any ergodic measure with positive entropy 
one can construct a horseshoe and use it to approximate its entropy, Lyapunov exponent and the integral of a continuous function. Namely, for each $i$
there exist $\beta_i>0$, a closed
interval $L_i$ and a family $\Omega_i$ of pairwise disjoint
closed intervals in the interior of $L_i$ such that:
\begin{itemize}
\item[(i)] for each $I\in\Omega_i$,
$f^{\beta_i}I=L_i$;

\item[(ii)] for any $x\in \bigcup_{I\in\Omega_i} I$ and $\psi\in\{\varphi,\log|Df|\}$, $\left|(1/\beta_i)
S_{\beta_i}\psi(x)
-\mu_i(\psi)\right|\leq1/i$;

\item[(iii)] $(1/\beta_i)\log\#\Omega_i\geq h(\mu_i)-1/i.$
\end{itemize}
We construct a family of intervals at smaller and smaller scales
which wander around different horseshoes.
By (A4), for each $i$ it is possible to choose $\gamma_i>0$ and a closed interval $\tilde L_i\subset L_i$ such that $f^{\gamma_i}$ sends $\tilde L_i$ homeomorphically onto $Y$.
Choose
a sequence $\{\kappa_i\}$ of positive integers inductively as follows.
Start with $\kappa_1=1$. Given $\kappa_{i-1}$, choose $\kappa_i$ to be a large integer which 
depends on
$\beta_1,\beta_2,\ldots,\beta_{i+1},\gamma_1,
\gamma_2,\ldots,\gamma_{i-1},\kappa_1,\kappa_2,\ldots,\kappa_{i-1},i,\sup|\varphi|,\alpha.$
Requirements among these constants will be made explicit at the end of the proof.

For each $k\geq1$, let $n=n(k)$, $s=s(k)$ be integers such that
$$
k=\kappa_1+\kappa_2+\cdots+\kappa_n+s\ \ \text{and}\ \ 0\leq s< \kappa_{n+1}.$$
Let
$$\Omega^{(k)}=\underbrace{\Omega_1\times \cdots\times\Omega_1}_{\kappa_1}
\times\underbrace{\Omega_2\times \cdots\times\Omega_2}_{\kappa_2} 
\times\cdots\times\underbrace{\Omega_{n}\times \cdots\times\Omega_{n}}_{\kappa_{n}} 
\times\underbrace{\Omega_{n+1}\times \cdots\times\Omega_{n+1}}_{s}.$$
Elements of $\Omega^{(k)}$ are denoted by $(I_1,\ldots,I_{k})$, i.e.,
$I_1\in\Omega_1,\ldots,I_{\kappa_1}\in\Omega_1,I_{\kappa_1+1}\in\Omega_2,$
and so on.

For each $k\geq1$ and $(I_1,\ldots,I_{k})\in\Omega^{(k)}$ we associate a closed
interval $[I_1,\ldots,I_{k}]$ inductively as follows.
Observe that $\Omega^{(1)}=\Omega_1$.
For each $I\in\Omega^{(1)}$,
define $[I]=I$.
Given $k\geq1$, $(I_1,\ldots,I_{k})\in\Omega^{(k)}$,
$[I_1,\ldots,I_{k}]$, $(I_1,\ldots,I_{k},I_{k+1})\in\Omega^{(k+1)}$,
define $[I_1,\ldots,I_{k},I_{k+1}]\subset[I_1,\ldots,I_{k}]$ by
$$[I_1,\ldots,I_{k},I_{k+1}]=\begin{cases}
&(f^{t}|[I_1,\ldots,I_{k}])^{-1}I_{k+1}\ \ \text{if}\ \ s<\kappa_{n+1}-1;\\
&(f^{t}|[I_1,\ldots,I_{k}])^{-1}((f^{\beta_{n+1}}|I_{k+1})^{-1}\tilde L_{n+1})\ \ \text{if} \ \ s
=\kappa_{n+1}-1,\end{cases}$$
where $t=t(k)$ is defined by
$$t=\beta_1\kappa_1+\gamma_1+\beta_2\kappa_2+\gamma_2+\cdots +
\beta_{n}\kappa_{n}+\gamma_{n} +\beta_{n+1}s.$$

Set
$$\mathcal F^{(k)}=\{[I_1,\ldots,I_{k}]\colon (I_1,\ldots,I_{k})\in\Omega^{(k)} \}.$$
This is a collection of pairwise disjoint closed intervals with the following properties:
if $s(k)>0$, then $f^{t(k)}[I_1,\ldots,I_{k}]=L_{n(k)+1}$;
if $s(k)=0$, then 
$f^{t(k)}$ sends $[I_1,\ldots,I_{k}]$ homeomorphically onto $Y$.
Observe that $\{\bigcup_{I\in\mathcal F^{(k)}}I\}_k$ is a nested sequence of closed sets.
Define
$$\Gamma=\bigcap_{k=1}^\infty\bigcup_{I\in\mathcal F^{(k)}}I.$$

Points in $\Gamma$  continue traveling from one horseshoe to the next generated by 
$\Omega_k$, $k\geq1$.
For the choice of $\{\kappa_i\}$ we will request
\begin{equation}\label{qn}
\kappa_i\gg\max\{\beta_1,\beta_2,\ldots,\beta_{i+1},\gamma_1,
\gamma_2,\ldots,\gamma_{i-1},\kappa_1,\kappa_2,\ldots,\kappa_{i-1}\}.\end{equation}
Then, generic finite orbits of $\Gamma$ spend most of their times near the last 
or the second last horseshoes,
and gain time averages in this duration.
As a result, the time averages along the finite orbits become nearly $\alpha$.
In fact, the following holds.

\begin{lemma}
$\Gamma\subset K_\varphi(\alpha)$.
\end{lemma}
\begin{proof}
Let $x\in\Gamma$. For a large integer $q$ let
$k\geq1$ be the maximal such that $t(k)\leq q$.
Then $q-t(k)\leq \beta_{n(k)+1}$.
Splitting the time interval $[0,q-1]$ is a concatenation of the duration around horseshoes and the transition between horseshoes,
and then applying (ii) to each of the corresponding orbit segments
we have
\begin{align*}|S_{q}\varphi(x)-q\alpha|\leq&
\sum_{j=0}^{\kappa_1-1}|S_{\beta_{1}}\varphi(f^{\beta_1j}x)-\beta_{1}\alpha|+
|S_{\gamma_{1}}\varphi(f^{\kappa_1\beta_1}x)-\gamma_{1}\alpha|\\
&+\sum_{j=0}^{\kappa_2-1}|S_{\beta_{2}}\varphi(f^{\beta_1\kappa_1+\gamma_1+\beta_2j}x)-\beta_{2}\alpha|
+|S_{\gamma_{2}}\varphi(f^{\beta_1\kappa_1+\gamma_1+\beta_2\kappa_2}x)-\gamma_{2}\alpha|+\cdots\\
&+\sum_{j=0}^{\kappa_{n}-1}|S_{\beta_{n}}\varphi(f^{\sum_{i=1}^{n-1}(\beta_i\kappa_i+\gamma_i)+\beta_{n}j}x)-\beta_{n}\alpha|
+|S_{\gamma_{n}}\varphi(f^{\sum_{i=1}^{n}\beta_i\kappa_i+\sum_{i=1}^{n-1}\gamma_i}x)-\gamma_{n}\alpha|\\
&+\sum_{j=0}^{s-1}|S_{\beta_{n+1}}\varphi(f^{\sum_{i=1}^{n}(\beta_i\kappa_i+\gamma_i)+
\beta_{n+1}j}x)-\beta_{n+1}\alpha|\\
&+|S_{q-t}\varphi(f^tx)-(q-t)\alpha|.
 \end{align*}
Using (ii) and the fact that $x$ is contained in an element of 
$\mathcal F^{(k)}$,  for every $2\leq\ell\leq n$
we have
\begin{align*}
\sum_{j=0}^{\kappa_{\ell}-1}|S_{\beta_{\ell}}\varphi(f^{\sum_{i=1}^{\ell-1}(\beta_i\kappa_i+\gamma_i)+\beta_{\ell}j}x)-\beta_{\ell}\alpha|\leq&
|S_{\beta_{\ell}}\varphi(f^{\sum_{i=1}^{\ell-1}(\beta_i\kappa_i+\gamma_i)+\beta_{\ell}j}x)-\beta_\ell\mu_{\ell}(\varphi)|\\
&+|\beta_{\ell}
\mu_{\ell}(\varphi)-\beta_{\ell}\alpha|\leq\frac{2\beta_{\ell}}{\ell}.
\end{align*}
Summing these and other reminder terms we get
\begin{align*}
|S_{q}\varphi(x)-q\alpha|&\leq\sum_{i=1}^n\gamma_{i}(\sup|\varphi|+\alpha)+\sum_{i=1}^n
\frac{2\beta_{i}
\kappa_{i}}{i}+\frac{2\beta_{n+1}s}{n+1}+(q-t)(\sup|\varphi|+\alpha)\\
&\leq\frac{3\beta_{n}
\kappa_{n}}{n}+\frac{2\beta_{n+1}s}{n+1}\leq\frac{5q}{n},
\end{align*}
where $\kappa_n$ is chosen sufficiently large so that the second inequality holds.
Since $n\to\infty$ as $q\to\infty$,
$x\in K_\varphi(\alpha)$ follows.
\end{proof}

For each $I\in\mathcal F^{(k)}$ choose a point $x_I\in I\cap\Gamma$ and define an atomic probability measure $\nu_k$ uniformly distributed on the set $\{x_I\colon I\in\mathcal F^{(k)}\}$.
Pick an accumulation point of the sequence $\{\nu_k\}$ and denote it by $\nu$. Since $\Gamma$ is closed we have $\nu(\Gamma)=1$.
For $x\in X$ and $\rho>0$, let $D_\rho(x)=\{y\colon |x-y|\leq\rho\}$. 
the inequality in Proposition \ref{lowd} follows from  \cite[Proposition 2.1]{You82} and the next

\begin{lemma}\label{mass}
For any $x\in \Gamma$ we have 
$$\varliminf_{\rho\to 0}\frac{\log \nu D_\rho(x)}{\log \rho} 
\geq \lim_{i\to\infty}\frac{h(\mu_i)}{\lambda(\mu_i)}.$$
\end{lemma}

%We finish the proof \eqref{HD} assuming the conclusion of this lemma.
%For arbitrary $\varepsilon>0$ and $c>0$, let $\rho_0>0$ denote the constant as in Lemma \ref%%{mass}.
%Consider a covering of $\Gamma$ by balls $\{B_{\rho_i}(x_i)\}_{i=1}^m$
%of radii $<\rho_0$. Then
%$$\nu_0(\Gamma)\leq\sum_{i=1}^m\nu_0(B_{\rho_i}(x_i))\leq c\sum_{i=1}^m\rho_i^{D_0-\varepsilon}.$$
%This implies that the Hausdorff $(D_0-\varepsilon)$-measure of $\Gamma$ is $\geq\frac{1}{c}$.
%Since $c>0$ is arbitrary, it follows that the Hausdorff $(D_0-\varepsilon)$-measure of $\Gamma$ %is $\infty$.
%Since $\varepsilon>0$ is arbitrary,
%we obtain $\dim_H(\Gamma)\geq D_0$.
%\begin{remark}
%{\rm \eqref{HD} also follows from [\cite{You82} Proposition 2.1].}
%\end{remark}
%\begin{lemma}\label{frostman}
%For any $x\in \Gamma$ we have
%$$\varliminf_{\rho\to0}\frac{\log\nu_0B_\rho(x)}{\log\rho}\geq\lim_{n\to\infty}\%frac{h(\mu_n)}{\lambda(\mu_n)}.$$
%\end{lemma}
%\eqref{HD} follows from Lemma \ref{frostman} and [\cite{You82} Proposition 2.1]

\begin{proof}
Consider the set of
pairs $(n,s)$ of integers such that
$n\geq0$ and $0\le s< \kappa_{n+1}$.
We introduce an order in this set as follows:
$(n_1,s_1)<(n_2,s_2)$ if $n_1<n_2$ or $n_1=n_2$ and $s_1<s_2$.
For a pair $(n,s)$ in this set, let
\begin{equation*}
a_{n,s}= \exp \left[-\beta_{n}\kappa_{n}\left(\lambda(\mu_n)+ \frac{2}{n}\right)
-\beta_{n+1}s\left(\lambda(\mu_{n+1})+ \frac{1}{n+1}\right) \right].\end{equation*}
Using \eqref{qn} 
it is easy to show that $a_{n+1,0}<a_{n,\kappa_{n+1}-1}$.
%\begin{claim}\label{decrease}
%$a_{n+1,0}<a_{n,q_{n+1}-1}$.
%\end{claim}
%\begin{proof}
%By definition we have
%\[
%a_{n+1,0}= \exp \left(-q_{n+1}k_{n+1}\left(\lambda(\mu_{n+1})+ \frac{2}{n+1}\right)\right),\]
%and
%\[
%a_{n,q_{n+1}-1}= \exp \left(-q_{n}k_{n}\left(\lambda(\mu_n)+
%\frac{2}{n}\right)
%-(q_{n+1}-1)k_{n+1}\left(\lambda(\mu_{n+1})+ \frac{1}{n+1}\right) \right).\]
%Using $q_{n+1}\gg\max\{q_n,k_n\}$ from \eqref{qn} 
%it is not hard to show that $a_{n+1,0}<a_{n,q_{n+1}-1}$.
%\end{proof}
%Hence, for any small
%$\rho>0$ either of the following holds: (i) there exist $(n,s),(n,s+1)\in\mathcal A$ such that $a_{n,s+1}\leq \rho\leq a_{n,s}$; (ii) there exist $(n+1,0),(n,q_{n+1}-1)\in \mathcal A$ such that $a_{n+1,0}\leq \rho\leq a_{n,q_{n+1}-1}$.
Hence
the sequence $\{a_{n,s}\}$ is monotone decreasing. 
Then for given small $\rho>0$ one can choose $k$ such that  
$ a_{n(k),s(k)} < \rho\leq a_{n(k-1),s(k-1)}$.

Let $I\in\mathcal F^{(k)}$.  
We have $\nu(\partial I)=0$, and for every $q\geq k$,
$$\nu_q(I)=\frac{\#\{J\in \mathcal F^{(q)}\colon J\subset I\}}{\#\mathcal F^{(q)}}=\frac{1}{\#\mathcal F^{(k)}}.$$
Hence
\begin{equation*}\label{nu0}\nu(I)=\lim_{q\to\infty}\nu_q(I)=\frac{1}{\#\mathcal F^{(k)}}.\end{equation*}

Using (ii) for $\psi=\log|Df|$ and (\ref{qn}), for all $x\in I$ we have
\begin{align*}
|Df^{t}(x)|&\leq\exp\left[\beta_{n}\kappa_{n}\left(\lambda(\mu_{n})+\frac{2}{n}\right)+\beta_{n+1}s
\left(\lambda(\mu_{n+1})+\frac{2}{n}\right)\right].
\end{align*}
Since $f^{t}I\subset X$,
the Mean Value Theorem gives
\begin{equation}\label{Q}
|I|\geq\frac{1}{2} \exp\left[-\beta_{n}\kappa_{n}\left(\lambda(\mu_{n})+\frac{2}{n}\right)-\beta_{n+1}s\left(\lambda(\mu_{n+1})+\frac{2}{n}\right)\right].
\end{equation}
Hence, for any $x\in \Gamma$,  
$D_\rho(x)$ intersects at most $2\exp\left[\beta_{n+1}s\left(2/n-1/(n+1)\right) \right]$-number
of elements of $\mathcal F^{(k)}$.
Using (iii) we have  \begin{align*}
\#\mathcal F^{(k)}&\geq (\#\Omega_{n})^{\kappa_{n}}\cdot(\#\Omega_{n+1})^{s}\geq\exp
\left[\beta_{n}\kappa_{n}\left(h(\mu_{n})-\frac{1}{n}\right)+
\beta_{n+1}s\left(h(\mu_{n+1})-\frac{1}{n+1}\right)\right],
\end{align*}
and therefore
\begin{align*}
\nu D_\rho(x)&\leq\frac{2}{\#\mathcal F^{(k)}}\exp\left[\beta_{n+1}s\left(\frac{2}{n}-\frac{1}{n+1}\right) \right]\\
&\leq2\exp\left[-\beta_{n}\kappa_{n}\left(h(\mu_{n})-\frac{1}{n}\right)-
\beta_{n+1}s\left(h(\mu_{n+1})-\frac{2}{n}\right)\right].\end{align*}
This yields
$$\frac{\log\nu D_\rho(x)}{\log\rho}\geq\frac{\beta_{n}\kappa_{n}\left(h(\mu_{n})-
1/n\right)+
\beta_{n+1}s\left(h(\mu_{n+1})-2/n\right)}
{\beta_{n}\kappa_{n}\left(\lambda(\mu_n)+2/n\right)
+\beta_{n+1}s\left(\lambda(\mu_{n+1})+1/(n+1)\right)}+\frac{\log2}{\log\rho}.$$
The desired inequality holds since $n \to\infty$ as $\rho\to0$.
\end{proof}

%$\frac{\log\nu D_{\rho}(x)}{\log\rho}\geq \zeta+\frac{\log c}{\log\rho}$.

\subsection{Continuity of the Birkhoff spectrum.}\label{continuity}

%We first consider the case where $\{\mu\in\mathcal M_f^e\colon0<|\mu(\varphi)-\alpha_0|<\varepsilon\}=\emptyset$ holds for some $\varepsilon>0$.
%Then for any $\alpha$ close to $\alpha_0$ we have
%$\{\mu\in\mathcal M_f^e\colon0<|\mu(\varphi)-\alpha|<\varepsilon/2\}=
%\emptyset$. Hence there is no ergodic $\mu$ with $\mu(\varphi)=\alpha_0$, for otherwise the second set is nonempty for $\alpha$ close to and differs from $\alpha_0$. This yields $B_\varphi(\alpha)=0$.
%The same reasoning gives $B_\varphi(\alpha)=0$ for any $\alpha$ close to $\alpha_0$, and thus the continuity holds at $\alpha_0$.
%We now consider the case where
%$\{\mu\in\mathcal M_f^e\colon0<|\mu(\varphi)-\alpha_0|<\varepsilon\}\neq\emptyset$ holds for %any $\varepsilon>0$.
%We show that the spectrum $\alpha\mapsto B_\varphi(\alpha)$ is continuous, and monotone in %the sense of Theorem A. 

From the formula in Theorem A, the spectrum is upper semi-continuous.
We argue by contradiction assuming that the spectrum is not lower semi-continuous at a point 
 $\alpha_0\in[c_\varphi,d_\varphi]$.
Then it is possible to choose $\epsilon_0>0$ and a monotone sequence $\{\alpha_n\}$ such that $\alpha_n\to\alpha_0$ and
\begin{equation}\label{lowdeq1}B_\varphi(\alpha_n)\leq B_\varphi(\alpha_0)-\epsilon_0.\end{equation} Let us suppose that $\{\alpha_n\}$ is monotone increasing.
Take $\mu_c\in\mathcal M_f$ with $\mu_c(\varphi)=c_\varphi$. The formula in Theorem A 
allows us to choose a sequence $\{\mu_k\}$ in $\mathcal M_f$ such that
$h(\mu_k)/\lambda(\mu_k)\geq B_\varphi(\alpha_0)-\epsilon_0/4$ and
$\mu_k(\varphi)\to\alpha_0$.
Choose a subsequence $\{\mu_{k(n)}\}$ such that
$\alpha_n\leq\mu_{k(n)}(\varphi)$. 
For each $n$ choose $0\leq t_n\leq1$ such that
$(1-t_n)\mu_c(\varphi)+t_n\mu_{k(n)}(\varphi)=\alpha_n$, and define
$\nu_n=(1-t_n)\mu_c+t_n\mu_{k(n)}$.
For all large $n$ we have 
$$B_\varphi(\alpha_n)=B_\varphi((1-t_n)\mu_c(\varphi)+t_n\mu_{k(n)}(\varphi))\geq \frac{h(\nu_n)}{\lambda (\nu_n)}\geq
B_\varphi(\alpha_0)- \epsilon_0 /2.$$
The second inequality follows from the linearity of entropies and Lyapunov exponents 
on measures, $t_n\to1$ and $\inf_{n}\lambda(\mu_{k(n)})\geq\lambda_{\inf}>0$.
This yields a contradiction to (\ref{lowdeq1}).
In the case where $\{\alpha_n\}$ is monotone decreasing,  take 
$\mu_d\in\mathcal M_f$ with $\mu_d(\varphi)=d_\varphi$ and use it in the place of $\mu_c$.

\section{Large deviation principle}
In this last section we prove Theorem B.
This amounts to proving the next proposition which gives an upper bound of 
deviation probabilities in terms of
the free energies of invariant measures. 

\begin{prop}\label{upper}
Let $d\geq1$,
$\varphi_1,\ldots,\varphi_d\in C(X)$ be Lipschitz, and let $\alpha_1,\ldots,\alpha_d\in\mathbb R$.
For any $\varepsilon>0$ there exists $n_0>0$ such that for every
$n\geq n_0$ there exists $\eta\in\mathcal
M_f$ such that:
\begin{equation}\label{upper1}
\frac{1}{n} \log \left|\left\{x\in\Lambda\colon
\frac{1}{n}S_n\varphi_j(x)\geq \alpha_j,\ j=1,\ldots,d\right\}\right|\leq
(1-\varepsilon)F(\eta)+4\sqrt{\varepsilon};\end{equation}
\begin{equation}\label{upper5}
\eta(\varphi_j) \geq \alpha_j-\sqrt{\varepsilon},\ \ \
j=1,\ldots,d.
\end{equation}
\end{prop}
We finish the proof of Theorem B assuming the conclusion of the proposition.
Recall that $M>0$ is such that
$f^{M}\Lambda=Y$. Let $\varepsilon_0>0$ be a small constant. For all large $n$ we have
$$\left\{x\in Y\colon \frac{1}{n}S_n\varphi_j(x)\geq \alpha_j \right\}\subset
f^{M}\left\{x\in\Lambda\colon \frac{1}{n}S_n\varphi_j(x)\geq
\alpha_j-\varepsilon_0\right\},$$ 
where
it is understood that $j$ runs over $\{1,2,\ldots,d\}$.
By the Mean Value Theorem,
\[ \left|\left\{x\in Y\colon \frac{1}{n}S_n\varphi_j(x)\geq \alpha_j\right\}\right| \leq 4^{M}\cdot
\left|\left\{x\in\Lambda\colon \frac{1}{n}S_n\varphi_j(x)\geq
\alpha_j-\varepsilon_0\right\}\right|.\] 
From this inequality and Proposition
\ref{upper} there exists $\eta\in\mathcal M_f$ such that
$\eta(\varphi_j) \geq \alpha_j-\varepsilon_0-\sqrt{\varepsilon}$ $(j=1,\ldots,d)$ and
\[\frac{1}{n}\log \left|\left\{x\in Y\colon
\frac{1}{n}S_n\varphi_j(x)\geq \alpha_j\right\}\right| \leq
\frac{M}{n}\log4+(1-C\varepsilon)F(\eta)+4\sqrt{\varepsilon}.\] Letting $n\to\infty$, and 
then $\varepsilon_0\to0$, $\varepsilon\to0$ we
get
\begin{equation}\label{upper-1}
\varlimsup_{n\to\infty}\frac{1}{n} \log \left|\left\{
\frac{1}{n}S_n\varphi_j\geq \alpha_j\right\}\right|
\leq\lim_{\varepsilon\to 0}\sup\left\{F(\nu)\colon\nu\in\mathcal
M_f,\ \nu(\varphi_j)\geq \alpha_j-\sqrt{\varepsilon}
\right\}.\end{equation} 
The lower large
deviations bound obtained in \cite{Chu11} gives
\begin{equation}\label{lower-1}
\varliminf_{n\to\infty}\frac{1}{n} \log \left|\left\{
\frac{1}{n}S_n\varphi_j> \alpha_j\right\}\right|
\geq\sup\left\{F(\nu)\colon\nu\in\mathcal M,\ \nu(\varphi_j)>
\alpha_j\right\},\end{equation} where $\sup\emptyset=-\infty$.
Theorem B follows from (\ref{upper-1})
(\ref{lower-1}) because the weak topology on
$\mathcal M$ has a countable base generated by open sets
of the form $ \left\{\nu\in\mathcal M \colon\nu(\varphi_j)>\alpha_j,\
j=1,\ldots,d \right\},$ where $d\geq1$, each $\varphi_j\in C(X)$ is
Lipschitz, and $\alpha_j\in\mathbb R$. \medskip

The rest of this section is devoted to the proof of Proposition
\ref{upper}. 
From the towers constructed in Sect.2 we extract horseshoes, and construct 
invariant measures supported on them 
with the properties as in the statement of the proposition.

\subsection{Construction of a horseshoe}\label{horseshoe}
Define 
$$\mathcal B_n=\left\{A\in\bigvee_{i=0}^{n-1} \hat f^{-i}\mathcal
D\colon A\subset\Delta_0,\ \ \frac{1}{n}S_n\varphi_j(x)\geq \alpha_j\ \
j=1,\ldots,d\ \text{ for some $x\in A$} \right\}.$$
Observe that
\begin{equation}\label{measure}
\left|\left\{x\in\Lambda\colon \frac{1}{n}S_n\varphi_j(x)
\geq \alpha_j\ \ j=1,\ldots,d\right\}\right|\leq |\mathcal
B_n|,\end{equation}
where $|\mathcal B_n|=\sum_{A\in\mathcal B_n}
|A|.$ To estimate this from above we use the next lemma,
the proof of which closely follows that of Lemma \ref{horse2}
with $\sigma$ replaced by $1$.

\begin{lemma}\label{horse}
For any $\varepsilon>0$ there exists $n''>0$ such that if $n\geq n''$ and $\mathcal B_n\neq\emptyset$ then there exist a finite
 collection $\mathcal L$ of pairwise disjoint closed intervals in $\Lambda$
 and an integer $r\in[(1-\varepsilon)n,(1+21\varepsilon/\lambda) n]$
 such that:

\begin{itemize}
\item[(a)] $\mathcal L$ generates a horseshoe for $f^{r}$;

\item[(b)] $\sum_{L\in\mathcal L} |L|\geq
e^{-3\sqrt{\varepsilon} n}|\mathcal B_n|$;

\item[(c)] for all $x\in H_{r}(\mathcal L)$, $(1/r)S_{r}\varphi_j(x)\geq
\alpha_j-\sqrt{\varepsilon}$, $j=1,\ldots,d$.
\end{itemize}
\end{lemma}
\begin{proof}
For each $B\in\mathcal B_n$ fix once and for all an interval $\tilde B$ and an integer $t=t_B$
for which the conclusions of
Lemma \ref{towerreturn} holds.
Let $\mathcal B_n(t)=\{B\in\mathcal B_n\colon t_B=t\}$.
Then 
$t_B\in[(1-\varepsilon)n, (1+19\varepsilon/\lambda)n]$.
Let $t_1$ be a value of $t$ which maximizes
$\sum_{B\in\mathcal B_n(t)}|\tilde B|$.
Then 
\begin{equation}\label{horse1}
\sum_{B\in\mathcal B_n(t_1)}|\tilde B|\geq\frac{1}{(1+20/\lambda)\varepsilon
n}\sum_{B\in\mathcal B_n}|\tilde B|.\end{equation} 
For each 
$B\in\mathcal B_n(t_1)$
define $L(B)$ to be the preimage of $I^+$ or $I^-$ under
$f^{t_1}|\tilde{B}$,
according to whether
$f^{t_1}\tilde{B}=
\Lambda^+$ or $=
\Lambda^-.$
Set $\mathcal L=\{L(B)\colon B\in\mathcal B_n(t_1)\}$ 
and $r=t_1+u$.
%Then $\mathcal L$ is a finite collection 
%of pairwise disjoint
%closed intervals in $\Lambda$, and
%$f^{r}$ sends each element of $\mathcal L$ diffeomorphically onto
%$\hat X$. 
The bounds on $r$ hold for sufficiently large $n$.
(\ref{horse1}) and Lemma \ref{towerreturn} implies
\begin{align*}\label{upper3}\sum_{L\in\mathcal L}|L|
%&\geq
%c\sum_{B\in\mathcal B_n(t_1)}|\tilde B|
%\geq
%\frac{c}{2C\varepsilon
%n}\sum_{B\in\mathcal
%B_n}|\tilde B|\geq\frac{ce^{-3\sqrt{\varepsilon}
%n}}{2C\varepsilon n }|\mathcal B_n|
\geq e^{-3\sqrt{\varepsilon}
n}|\mathcal B_n|
.\end{align*}

To prove (c), for each $B\in\mathcal B_n(t_1)$ pick $x_B\in B$
such that
$S_n\varphi(x_B)\geq \alpha_jn$ for $j=1,\ldots,d$. We have 
%$
%|S_r\varphi_j(x_B)-S_n\varphi_j(x_B)|\leq \sup|\varphi_j|\cdot|r-n|$, and so
\begin{equation}\label{con1}
S_r\varphi_j(x_B)\geq S_n\varphi_j(x_B)-
\sup|\varphi_j|\cdot|r-n|\geq \alpha_jn-(21\varepsilon/\lambda) n\geq
\left(\alpha_j-\sqrt{\varepsilon}/2\right)r.\end{equation} By
Sublemma \ref{bdd}, for any $x\in B$ we have
$|S_{t_1}\varphi_j(x_B)-S_{t_1}\varphi_j(x)|\leq {\rm Lip}(\varphi_j)\cdot
C\delta^{-1}$, and thus
\begin{equation}\label{con2}
|S_r\varphi_j(x_B)-S_r\varphi_j(x)|\leq {\rm
Lip}(\varphi_j)\cdot
C\delta^{-1}+2\sup|\varphi_j|\cdot(r-t_1)\leq \sqrt{\varepsilon}r/2.\end{equation} (\ref{con1}) (\ref{con2}) yield
$S_{r}\varphi_j(x)\geq (\alpha_j-\sqrt{\varepsilon})r$.
\end{proof}

\subsection{Construction of a measure on the horseshoe}\label{construct1}
We construct a measure $\eta$ for which (\ref{upper1})
(\ref{upper5}) hold. 
For sufficiently large $n$ with $\mathcal B_n\neq\emptyset$, 
choose a finite collection $\mathcal L$ of pairwise
disjoint closed intervals in $\Lambda$ and a positive integer $r$ for which 
the conclusions of Lemma \ref{horse} hold. 
 %Let $\nu_{\Psi}$ denote the equilibrium state of $G$
%for the potential $-\Psi$, namely a $G$-invariant measure such that
%$$h_G(\nu_{\Psi})-\nu_{\Psi}(\Psi)=\sup\left\{h_G(\nu)-\nu(\Psi)
%\colon\text{$\nu$ is $G$-invariant}\right\}.$$ 
%Write $\mathcal L=\{L_1,\ldots,L_{\sharp \mathcal L}\}$ and
Set $G=f^r$.
The argument in Sect.\ref{construct2} shows that there exists 
a $G$-invariant probability measure $\nu_\infty$ supported on
$H_r(\mathcal L)$
satisfying
\begin{equation*}
%\label{lem1'}
h_G(\nu_\infty)-\int\log|DG| d\nu_\infty
\geq \log\sum_{L\in \mathcal L}|L|-\log \kappa ,
\end{equation*}
where $h_G(\nu_\infty)$
denotes the entropy of $(G,\nu_\infty)$.
Define $\eta\in\mathcal M_f$ 
by $\eta=(1/r)\sum_{i=0}^{r-1}(f^i)_*\nu_{\infty}$.
From Lemma \ref{horse}(c) it follows
that $\eta(\varphi_j)\geq
\alpha_j-\sqrt{\varepsilon}$, and (\ref{upper5}) holds.
Since $F(\eta)\leq0$ and
$r\geq (1-\varepsilon)n$,
using Lemma \ref{horse}(b)
we have
\begin{align*}
n\cdot F(\eta)
&\geq \frac{r}{1-\varepsilon} F(\eta) =\frac{1}{1-\varepsilon}\left(
h_G(\nu_\infty)-\int\log|DG| d\nu_\infty\right) \\
&\geq \frac{1}{1-\varepsilon}\left(
\log\sum_{L\in \mathcal L}|L|-\log \kappa \right) \geq \frac{1}{1-\varepsilon}\left( 
\log |\mathcal B_n| -4\sqrt{\varepsilon} n \right) .
\end{align*}
Rearranging this and using \eqref{measure} yields 
\begin{align*} 
\frac{1}{n}\log \left|\left\{x\in\Lambda\colon \frac{1}{n}S_n\varphi_j(x) 
\geq \alpha_j\ \ j=1,\ldots,d\right\}\right| \leq \frac{1}{n}\log |\mathcal B_n| \leq(1-\varepsilon)F(\eta)+4\sqrt{\varepsilon}.
\end{align*}
Hence (\ref{upper1}) holds.

\subsection*{Acknowledgments}
We thank anonymous referees for very useful comments.
The first-named author is partially supported by the Grant-in-Aid 
for Scientific Research (C) of the JSPS, Grant No. 24540212.
The second-named author is partially supported by the Grant-in-Aid
for Young Scientists (B) of the JSPS, Grant No. 23740121.

\bibliographystyle{amsplain}

\end{document}